\documentclass[10pt]{amsart}
\usepackage{hyperref,amssymb}
\usepackage[english]{babel}
\usepackage{fancyhdr}
\usepackage{colortbl}
\usepackage{enumerate}
\usepackage{graphicx}
\usepackage{float}
\usepackage{mathrsfs}
\usepackage[left=2.5cm,right=2.5cm,top=2.75cm,bottom=2.75cm]{geometry}
\usepackage{tikz}
\usetikzlibrary{arrows,automata,shapes,calc,patterns}
\usepackage{cite}
\usepackage{xcolor}

\usepackage[colorinlistoftodos]{todonotes}
\definecolor{blue}{rgb}{0,0,1}

\newtheorem{theorem}{Theorem}[section]
\newtheorem{prop}[theorem]{Proposition}
\newtheorem{lemma}[theorem]{Lemma}
\newtheorem{cor}[theorem]{Corollary}

\theoremstyle{definition}
\newtheorem{definition}{Definition}[section]

\theoremstyle{remark}
\newtheorem{remark}{Remark}[section]

\numberwithin{equation}{section}

\newcommand{\vertiii}[1]{{\left\vert\kern-0.25ex\left\vert\kern-0.25ex\left\vert #1
    \right\vert\kern-0.25ex\right\vert\kern-0.25ex\right\vert}}

\def \L {\mathcal{L}} 
\def \N {\mathbb{N}}  
\def \R {\mathbb{R}}  

\def\RN {\mathbb{R}^{N}} 

\def \Linfty {L^{\infty}(\Omega)}

\def \sLap {(-\Delta)^{s}}
\def \gradu {\nabla u}
\def \gradv {\nabla v}

\newcommand{\supp}{\operatorname{Supp}}
\renewcommand{\epsilon} {\varepsilon}

\newcommand{\esssup}{\operatorname{esssup}}

\newcommand{\dist}{\operatorname{dist}}

\newcommand{\D }{\Delta }
\newcommand{\e }{\varepsilon }

\renewcommand{\l }{\lambda }
\renewcommand{\L }{\Lambda }

\newcommand{\n }{\nabla }

\renewcommand{\O }{\Omega }

\newcommand{\re}{\mathbb{R}}
\newcommand{\ren}{\re^N}

\newcommand{\dyle}{\displaystyle}

\newcommand{\irn}{\int_{\re^N}}
\newcommand{\io}{\int\limits_\O}

\hypersetup{
colorlinks=true, breaklinks=true, bookmarks=true,bookmarksnumbered,
urlcolor=green!50!black, linkcolor=blue!85!black, citecolor=red!85!black, 
pdftitle={}, 
pdfauthor={\textcopyright}, 
pdfsubject={}, 
pdfkeywords={}, 
pdfcreator={pdfLaTeX}, 
pdfproducer={LaTeX with hyperref and ClassicThesis} 
}

\begin{document}

\title[Nonlinear fractional Laplacian problems with nonlocal ``gradient terms'']{Nonlinear fractional Laplacian problems with nonlocal ``gradient terms''}


\author[Boumediene Abdellaoui and Antonio J. Fern\'andez]{}
\address{}
\curraddr{}
\email{}
\thanks{}





\date{}

\dedicatory{}

\maketitle

\vspace{-0.03cm}
\centerline{\scshape Boumediene Abdellaoui}
\smallskip
{\footnotesize
 \centerline{Laboratoire d'Analyse Nonlin\'eaire et Math\'ematiques Appliqu\'ees, D\'epartement de Math\'ematiques,}
 \centerline{Universit\'e Abou Bakr Belka\"id, Tlemcen, 13000, Algeria}
  \vspace{0.05cm}
\centerline{\textit{E-Mail address} : \texttt{boumediene.abdellaoui@inv.uam.es}}
}

\bigskip

\centerline{\scshape Antonio J. Fern\'andez}
\smallskip
{\footnotesize
 \centerline{Univ. Valenciennes, EA 4015 - LAMAV - FR CNRS 2956, F-59313 Valenciennes, France}
 \vspace{0.09cm}
   \centerline{Laboratoire de Math\'ematiques (UMR 6623), Universit\'e de Bourgogne-Franche-Comt\'e,}
 \centerline{16 route de Gray, 25030 Besan\c con Cedex, France}
 \vspace{0.05cm}
\centerline{\textit{E-Mail address} : \texttt{antonio\_jesus.fernandez\_sanchez@univ-fcomte.fr}}
} 

\begin{center}\rule{1\textwidth}{0.1mm} \end{center}
\begin{abstract}
Let $\Omega \subset \RN$, $N \geq 2$, be a smooth bounded domain. For $s \in (1/2,1)$, we consider a problem of the form
\begin{equation*} 
\left\{
\begin{aligned}
(-\Delta)^s u & = \mu(x)\, \mathbb{D}_s^{2}(u) + \lambda f(x)\,, & \quad \textup{ in } \Omega,\\
u & = 0\,, & \quad \textup{ in } \RN \setminus \Omega,
\end{aligned}
\right.
\end{equation*}
\noindent where $\lambda > 0$ is a real parameter, $f$ belongs to a suitable Lebesgue space, $\mu \in \Linfty$ and $\mathbb{D}_s^2$ is a nonlocal ``gradient square'' term given by
\begin{equation*}
\mathbb{D}_s^2 (u) =  \frac{a_{N,s}}{2}\textup{ p.v.}  \int_{\RN} \frac{|u(x)-u(y)|^2}{|x-y|^{N+2s}} dy \,.
\end{equation*}
Depending on the real parameter $\lambda > 0$, we derive existence and non-existence results. The proof of our existence result relies on sharp Calder\'on-Zygmund type regularity results for the fractional Poisson equation with low integrability data. We also obtain existence results for related problems involving different nonlocal diffusion terms.
\end{abstract}

\begin{center} \rule{1 \textwidth}{0.1mm} \end{center}

\section{Introduction and main results} \label{1}

In the last fifteen years, there has been an increasing interest in the study of partial differential equations involving integro-differential operators. In particular, the case of the fractional Laplacian has been widely studied and is nowadays a very active field of research. This is due not only to its mathematical richness. The fractional Laplacian has appeared in a great number of equations modeling real world phenomena, especially those which take into account nonlocal effects. Among others, let us mention applications in quasi-geostrophic flows \cite{C_V_2010}, quantum mechanic \cite{L_2000}, mathematical finances \cite{Apple_2004, C_T_2004}, obstacle problems \cite{B_F_RO_2018, B_F_RO_2018-2, C_F_2013} and crystal dislocation \cite{DP_P_V_2015, DP_F_V_2014, T_1997}.

\medbreak
The first aim of the present paper is to discuss, depending on the real parameter $\lambda > 0$, the existence and non-existence of solutions to the Dirichlet problem
\[ \tag{$P_{\lambda}$} \label{Plambda}
\left\{
\begin{aligned}
(-\Delta)^s u & = \mu(x)\, \mathbb{D}_s^2(u) + \lambda f(x)\,, & \quad \textup{ in } \Omega,\\
u & = 0\,, & \quad \textup{ in } \RN \setminus \Omega,
\end{aligned}
\right.
\]
under the assumption
\[ \tag{$A_1$}  \label{A1}
\left\{
\begin{aligned}
& \Omega \subset \RN,\ N \geq 2, \textup{ is a bounded domain with }\partial \Omega \textup{ of class } \mathcal{C}^{2}, \\
& s \in (1/2,1),\\
& f \in L^m(\Omega) \textup{ for some } m > N/2s \textup{ and } \mu \in \Linfty.\\
\end{aligned}
\right.
\] 
Throughout the work, $\sLap$ stands for the, by know classical, fractional Laplacian operator. For a smooth function $u$ and $s \in (0,1)$, it can be defined as
\[ \sLap u(x):= a_{N,s}\textup{ p.v.}\int_{\RN} \frac{u(x)-u(y)}{|x-y|^{N+2s}} dy,\]
where
\[ a_{N,s} := \left( \int_{\RN} \frac{1-\cos(\xi_1)}{|\xi|^{N+2s}} d \xi \right)^{-1} = -\, \frac{ 2^{2s} \Gamma \left( \frac{N}{2}+s \right)}{ \pi^{\frac{N}{2}} \Gamma(-s) },\]
is a normalization constant and ``p.v.'' is an abbreviation for ``in the principal value sense''. In \eqref{Plambda}, $\mathbb{D}_s^2$ is a nonlocal diffusion term. It plays the role of the ``gradient square'' in the nonlocal case and is given by
\begin{equation} \label{nonlocal gradient}
 \mathbb{D}_s^2 (u) =  \frac{a_{N,s}}{2}\textup{ p.v.} \int_{\RN} \frac{|u(x)-u(y)|^2}{|x-y|^{N+2s}} dy \,.
\end{equation}
Since they will not play a role in this work, we normalize the constants appearing in the definitions of $\sLap$ and $\mathbb{D}_s^2$ and we omit the p.v. sense. However, let us stress that these constants guarantee
\begin{equation} \label{limit fractional Laplacian}
 \lim_{s \to 1^{-}} \sLap u(x) = -\Delta u (x), \quad \forall\ u \in \mathcal{C}_0^{\infty}(\RN),
\end{equation}
and
\begin{equation} \label{limit gradient term}
\lim_{s \to 1^{-}} \mathbb{D}_s^2(u(x))= |\nabla u(x)|^2, \quad \forall\ u \in \mathcal{C}_0^{\infty}(\RN).
\end{equation}
We refer to \cite{DN_P_V_2012} and \cite{C_D_2018} respectively for a proof of \eqref{limit fractional Laplacian} and \eqref{limit gradient term}. Hence, at least formally, if $s \to 1^{-}$ in \eqref{Plambda}, we recover the local problem
\begin{equation} \label{local}
\left\{
\begin{aligned}
-\Delta u & = \mu(x) |\nabla u |^{2} + \lambda f(x), \quad & \textup{ in } \Omega, \\
u & = 0, & \textup{ on } \partial \Omega.
\end{aligned}
\right.
\end{equation}
This equation corresponds to the stationary case of the Kardar-Parisi-Zhang model of growing interfaces introduced in \cite{K_P_Z_1986}. The existence and multiplicity of solutions to problem \eqref{local} and of its different extensions have been extensively studied and it is still an active field of research. See for instance \cite{A_DA_P_2006, F_M_2000, G_M_P_2006, B_M_P_1992, A_DC_J_T_2015, DC_F_2018}. In most of these papers, the existence of solutions is proved using either a priori estimates or, when it is possible, a suitable change of variable to obtain an equivalent semilinear problem. However, neither of these techniques seem to be appropriate to deal with the nonlocal case \eqref{Plambda}.
\medbreak
Let us also point out that in \cite{C_V_2014-JFA}, using pointwise estimates on the Green function for the fractional Laplacian, the authors deal with the nonlocal-local problem
\begin{equation} \label{nonv}
\left\{
\begin{aligned}
(-\Delta)^s u & =|\nabla u |^q + \lambda f(x), \quad & \textup{ in } \Omega, \\
u & = 0, & \textup{ in }  \RN \setminus \Omega.
\end{aligned}
\right.
\end{equation}
For $s \in (1/2,1)$, $1 < q  < \frac{N}{N-(2s-1)}$, $f \in L^1(\Omega)$ and $\lambda > 0$ small enough they obtained the existence of a solution to \eqref{nonv}. This existence result was later completed in \cite{A_P_2018} where, under suitable assumptions on $f$, the authors showed the existence of a solution to \eqref{nonv} for all $1 < q < \infty$ and $\lambda > 0$ small enough.

\medbreak
Following \cite{C_V_2014-JFA,C_V_2014-JDE} we introduce the following notion of weak solution to \eqref{Plambda}.

\begin{definition} \label{weak Sol Plambda}
We say that $u$ is a\textit{ weak solution} to \eqref{Plambda} if $u$ and $\mathbb{D}_s^2(u)$ belong to $L^1(\Omega)$, $u \equiv 0$ in $\mathcal{C}\Omega := \RN \setminus \Omega$ and
\begin{equation}
\int_{\Omega} u \sLap \phi\, dx = \int_{\Omega} \big( \mu(x)\mathbb{D}_s^2(u)+ \lambda f(x) \big) \phi\, dx, \quad \forall\  \phi \in \mathbb{X}_s,
\end{equation}
where
\begin{equation} \label{Xs}
 \mathbb{X}_s := \Big\{ \xi \in \mathcal{C}(\RN): \supp{\xi} \subset \overline{\Omega},\ \sLap \xi(x) \textup{ exists }\forall\ x  \in \Omega \textup{ and } |\sLap \xi(x)| \leq C \textup{ for some } C > 0\, \Big\}.
\end{equation}
\end{definition}

In the spirit of the existing results for the local case, our first main result shows the existence of a weak solution to \eqref{Plambda} under a smallness condition on $\lambda f$.

\begin{theorem} \label{main th existence}
Assume that \eqref{A1} holds. Then there exists $\lambda^{\ast} > 0$ such that, for all $0 < \lambda \leq \lambda^{\ast}$, \eqref{Plambda} has a weak solution $u \in W_0^{s,2}(\Omega) \cap \mathcal{C}^{0,\alpha}(\Omega)$ for some $\alpha > 0$.
\end{theorem}

\begin{remark} $ $
\begin{itemize}
\item[a)]The definition of $W_0^{s,2}(\Omega)$ will be introduced in Section \ref{2}.
\item[b)] In 1983, L. Boccardo, F. Murat and J.P. Puel \cite{B_M_P_1983} already pointed out that the existence of solution to \eqref{local} is not guaranteed for every $\lambda f \in \Linfty$. Some extra conditions are needed. Hence, the smallness condition appearing in Theorem \ref{main th existence} was somehow expected.
\item[c)] For $\lambda f \equiv 0$, $u \equiv 0$ is a solution to \eqref{Plambda} that obviously belongs to $W_0^{s,2}(\Omega) \cap \mathcal{C}^{0,\alpha}(\Omega)$. Hence, there is no loss of generality to assume that $\lambda > 0$.
\end{itemize}
\end{remark}

The counterpart of $|\nabla u|^2$ in \eqref{local} is played in \eqref{Plambda} by $\mathbb{D}_s^2(u)$. This term appears in several applications. For instance, let us mention \cite{M_S_2015, C_D_2018, Schik_2015} where it naturally appears as the equivalent of $|\gradu|^2$ when considering fractional harmonic maps into the sphere.

\medbreak
Let us now give some ideas of the proof of Theorem \ref{main th existence}. As in the local case, see for instance \cite{P_2014}, the existence of solutions to \eqref{Plambda} is related to the regularity of the solutions to a linear equation of the form
\begin{equation} \label{linearEq introduction}
\left\{
\begin{aligned}
\sLap v & = h(x), \quad &  \textup{ in } \Omega,\\
v & = 0, & \textup{ in } \RN \setminus \Omega.
\end{aligned}
\right.
\end{equation}
In Section \ref{3}, we obtain sharp Calder\'on-Zygmund type regularity results for the fractional Poisson equation \eqref{linearEq introduction} with low integrability data. We believe these results are of independent interest and will be useful in other settings. Actually, Section \ref{3} can be read as an independent part of the present work.  In particular, we refer the interested reader to Propositions \ref{regularity}, \ref{corollary regularity-2} and \ref{corollary regularity-3}.
\medbreak
Having at hand suitable regularity results for \eqref{linearEq introduction} and inspired by \cite[Section 6]{P_2014}, we develop a fixed point argument to obtain a solution to \eqref{Plambda}. Note that, due to the nonlocality of the operator and of the ``gradient term'', the approach of \cite{P_2014} has to be adapted significantly. In particular, the form of the set where we apply the fixed point argument seems to be new in the literature. We consider a subset of $W_0^{s,1}(\Omega)$ where, in some sense, we require more ``differentiability'' and more integrability. This extra ``differentiability'' is a purely nonlocal phenomena and it is related with our regularity results for \eqref{linearEq introduction}. See Section \ref{4} for more details.
\medbreak
Let us also stress that the restriction $s \in (1/2,1)$ comes from the regularity results of Section \ref{3}. If suitable regularity results for \eqref{linearEq introduction} with $s \in (0,1/2]$ were available, our fixed point argument would provide the desired existence results to \eqref{Plambda}. See Sections \ref{3} and \ref{7} for more details.

\medbreak
Next, let us prove that the smallness condition imposed in Theorem \ref{main th existence} is somehow necessary.

\begin{theorem} \label{main th non-existence}
Assume \eqref{A1} and suppose that $\mu(x) \geq \mu_1 > 0$ and $f^{+} \not \equiv 0$. Then there exists $\lambda^{\ast \ast} > 0$ such that, for all $\lambda > \lambda^{\ast \ast}$, \eqref{Plambda} has no weak solutions in $W_0^{s,2}(\Omega)$.
\end{theorem}

\medbreak
\begin{remark} $ $
\begin{itemize}
\item[a)]Observe that, if $v$ is a solution to
\[
\left\{
\begin{aligned}
(-\Delta)^s v & = -\mu(x)\,\mathbb{D}_s^2(v) - \lambda f(x)\,, & \quad \textup{ in } \Omega,\\
v & = 0\,, & \quad \textup{ in } \RN \setminus \Omega,
\end{aligned}
\right.
\]
then $u = -v$ is a solution to \eqref{Plambda}. Hence, if $\mu(x) \leq -\mu_1 < 0$ and $f^{-} \not \equiv 0$ we recover the same kind of non-existence result and the smallness condition is also required. \smallbreak
\item[b)] Since we do not use the regularity results of Section \ref{3}, the restriction $s \in (1/2,1)$ is not necessary in the proof of Theorem \ref{main th non-existence}. The result holds for all $s \in (0,1)$.
\end{itemize}
\end{remark}

Also, in order to show that the regularity imposed on the data $f$ is almost optimal, we provide a counterexample to our existence result when the regularity condition on $f$ is not satisfied. The proof makes use of the Hardy potential.

\begin{theorem}\label{optimality}
Let $\Omega \subset \RN,\ N \geq 2,$  be a bounded domain with $\partial \Omega$  of class $\mathcal{C}^{2}$, let $s \in (0,1)$ and let $\mu \in \Linfty$ such that $\mu(x) \geq \mu_1 > 0$. Then, for all $1 \leq p < \frac{N}{2s}$, there exists $f \in L^p(\Omega)$ such that \eqref{Plambda} has no weak solutions in $W_0^{s,2}(\Omega)$ for any $\lambda > 0$.
\end{theorem}
\medbreak
Using the same kind of approach than in Theorem \ref{main th existence}, i.e. regularity results for \eqref{linearEq introduction} and our fixed point argument, one can obtain existence results for related problems involving different nonlocal diffusion terms and different nonlinearities. 
\medbreak
First, we deal with the Dirichlet problem
\begin{equation}\label{map1} \tag{$\widetilde{P}_{\lambda}$}
\left\{
\begin{aligned}
(-\Delta)^s u & = \mu(x)\, u\, \mathbb{D}_s^2(u) + \lambda f(x)\,, & \quad \textup{ in } \Omega,\\
u & = 0\,, & \quad \textup{ in } \RN \setminus \Omega.
\end{aligned}
\right.
\end{equation}
For $\mu(x) \equiv 1$, this problem can be seen as a particular case of the fractional harmonic maps problem considered in \cite{C_D_2018, M_S_2015}.

\begin{remark}
The notion of weak solution to \eqref{map1} is essentially the same as in Definition \ref{weak Sol Plambda}. The only difference is that we now require that $u$ and $u\, \mathbb{D}_s^2(u)$ belong to $L^1(\Omega)$.
\end{remark}

We derive the following existence result for $\lambda f$ small enough.

\begin{theorem} \label{main th existence map1}
Assume that \eqref{A1} holds. Then, there exists $\lambda^{\ast} > 0$ such that, for all $0 < \lambda \leq \lambda^{\ast}$, \eqref{map1} has a weak solution $u \in W_0^{s,2}(\Omega) \cap \mathcal{C}^{0,\alpha}(\Omega)$ for some $\alpha > 0$.
\end{theorem}

Next, motivated by some other results on fractional harmonic maps into the sphere \cite{DL_S_2014, DL_R_2011} and some classical results of harmonic analysis \cite[Chapter V]{Stein_book_1970}, we consider a different diffusion term. Depending on the real parameter $\lambda > 0$, we study the existence of solutions to the Dirichlet problem
\[ \tag{$Q_{\lambda}$} \label{Qlambda}
\left\{
\begin{aligned}
(-\Delta)^s u & = \mu(x)|(-\Delta)^{\frac{s}{2}} u|^q + \lambda f(x)\,, & \quad \textup{ in } \Omega,\\
u & = 0\,, & \quad \textup{ in } \RN \setminus \Omega,
\end{aligned}
\right.
\]
under the assumption
\[ \tag{$B_1$}  \label{B1}
\left\{
\begin{aligned}
& \Omega \subset \RN,\ N \geq 2, \textup{ is a bounded domain with }\partial \Omega \textup{ of class } \mathcal{C}^{2}, \\
& f \in L^m(\Omega) \textup{ for some } m \geq 1 \textup{ and } \mu \in \Linfty,\\
& s \in (1/2,1)\,  \textup{ and }\, 1 < q < \frac{N}{N-ms}.
\end{aligned}
\right.
\]

\begin{remark}
If $m \geq N/s$, we just need to assume $1 < q < \infty$ in \eqref{B1}.
\end{remark}

Since the diffusion term considered in \eqref{Qlambda} is different from the ones in \eqref{Plambda} and \eqref{map1}, we shall make precise the notion of weak solution to \eqref{Qlambda}.

\begin{definition} \label{weak Sol Qlambda}
We say that $u$ is a\textit{ weak solution} to \eqref{Qlambda} if $u \in L^1(\Omega)$, $|(-\Delta)^{\frac{s}{2}}u| \in L^q(\Omega)$, $u \equiv 0$ in $\mathcal{C}\Omega$ and
\begin{equation}
\int_{\Omega} u \sLap \phi\, dx = \int_{\Omega} \big( \mu(x) |(-\Delta)^{\frac{s}{2}} u|^q + \lambda f(x) \big) \phi\, dx, \quad \forall\  \phi \in \mathbb{X}_s,
\end{equation}
where $\mathbb{X}_s$ is defined in \eqref{Xs}.
\end{definition}

\begin{theorem} \label{main th Qlambda}
Assume that \eqref{B1} holds. Then there exists $\lambda^{\ast} > 0$ such that, for all $0 < \lambda \leq \lambda^{\ast}$,  \eqref{Qlambda} has a weak solution $u \in W_0^{s,1}(\Omega)$.
\end{theorem}

\begin{remark}
The regularity results for \eqref{linearEq introduction} that we need to prove Theorem \ref{main th Qlambda} are different from the ones used in Theorems \ref{main th existence} and \ref{main th existence map1}. Nevertheless, the restriction $s \in (1/2,1)$ still arises out of these regularity results. See Proposition \ref{regu-grad1} for more details.
\end{remark}

\medbreak
Finally, for $s \in (0,1)$ and $\phi \in \mathcal{C}_0^{\infty}(\RN)$, following \cite{S_S_2015, Ponce_Book_2015}, we define the (distributional Riesz) \textit{fractional gradient of order $s$} as the vector field $\nabla^s: \RN \to \R$ given by
\begin{equation}\label{frac-g}
\n^s \phi(x):=\irn \frac{\phi(x)-\phi(y)}{|x-y|^s}\frac{x-y}{|x-y|}\frac{dy}{|x-y|^N}, \quad \forall\ x \in \RN.
\end{equation}
Then we deal with the Dirichlet problem
\begin{equation} \label{Rlambda} \tag{$\widetilde{Q}_{\lambda}$}
\left\{
\begin{aligned}
(-\Delta)^s u & = \mu(x)|\n^s u|^q + \lambda f(x)\,, & \quad \textup{ in } \Omega,\\
u & = 0\,, & \quad \textup{ in } \RN \setminus \Omega.
\end{aligned}
\right.
\end{equation}

\begin{remark}
The notion of weak solution to \eqref{Rlambda} has to be understood as in Definition \ref{weak Sol Qlambda}.
\end{remark}

\begin{theorem} \label{maint th Rlambda}
Assume that \eqref{B1} holds. Then there exists $\lambda^{\ast} > 0$ such that, for all $0 < \lambda \leq \lambda^{\ast}$, \eqref{Rlambda} has a weak solution $u \in W_0^{s,1}(\Omega)$.
\end{theorem}

\medbreak

We end this section describing the organization of the paper. In Section \ref{2}, we introduce the suitable functional setting to deal with our problems and we also recall some known results that will be useful. In Section \ref{3}, which is independent of the rest of the work, we prove Calder\'on-Zygmund type regularity results for the fractional Poisson equation \eqref{linearEq introduction}. Section \ref{4} is devoted to the proofs of Theorems \ref{main th existence} and \ref{main th existence map1}. Section \ref{5} contains the proofs of Theorems \ref{main th non-existence} and \ref{optimality}. Section \ref{6} deals with \eqref{Qlambda} and \eqref{Rlambda}, i.e., it is devoted to the proofs of Theorems \ref{main th Qlambda} and \ref{maint th Rlambda}. Finally, in Section \ref{7}, we present some remarks and open problems.

\bigbreak

\noindent \textbf{Acknowledgments.} Part  of  this  work  was done  while  the  first author was visiting the mathematics department of the University Bourgogne Franche-Comt\'e. He would like to thank the LMB for the warm hospitality and financial support. The second author thanks Prof. Tommaso Leonori for stimulating discussions concerning the subject of the present work.

\bigbreak

\noindent \textbf{Notation.}
\begin{enumerate}{\small
\item[1)] In  $\mathbb R^N$, we use the notations $|x|=\sqrt{x_1^2+\ldots+x_N^2}$ and $B_R(y)=\{x\in \mathbb R^N : |x-y|<R\}$.
\item[2)] For a bounded open set $\Omega \subset \RN$ we denote its complementary as $\mathcal{C}\Omega$, i.e. $\mathcal{C}\Omega = \RN \setminus \Omega$.
\item[3)]
For $p \in (1,\infty),$ we denote by $p^{\prime}$ the  conjugate exponent of $p$, namely $p^{\prime} = p/(p-1)$ and by $p_s^*$ the Sobolev critical exponent i.e.
$p_s^*=\frac{Np}{N-sp}$ if $sp<N$ and $p_s^*=+\infty$ in case $sp\geq N$.
}
\item[4)] For $u \in \Linfty$ we use the notation $\|u\|_{\infty} = \|u\|_{L^{\infty}(\Omega)} = \esssup_{x \in \Omega} |u(x)|$.
\end{enumerate}
\medbreak

\section{Functional setting and Useful tools} \label{2}

In this section we present the functional setting and some auxiliary results that will play an important role throughout the paper. We begin recalling the definition of the fractional Sobolev space.

\begin{definition}
Let $\Omega$ be an open set in $\RN$ and $s \in (0,1)$. For any $p \in [1,\infty)$, the fractional Sobolev space $W^{s,p}(\Omega)$ is defined as
\[ W^{s,p}(\Omega):= \left\{ u \in L^p(\Omega): \iint_{\Omega \times \Omega} \frac{|u(x)-u(y)|^p}{|x-y|^{N+sp}}dx dy < \infty \right\}. \]
It is a Banach space endowed with the usual norm
\[ \|u\|_{W^{s,p}(\Omega)} := \left( \|u\|_{L^p(\Omega)}^p + \iint_{\Omega \times \Omega} \frac{|u(x)-u(y)|^p}{|x-y|^{N+sp}} dx dy \right)^{\frac{1}{p}}.\]
\end{definition}

Having at hand this definition we introduce the suitable space to deal with our problems. 

\begin{definition}
Let $\Omega \subset \RN$ be a bounded domain with boundary $\partial \Omega$ of class $\mathcal{C}^{0,1}$ and $s \in (0,1)$. For any $p \in [1, \infty)$. We define the space $W_0^{s,p}(\Omega)$ as
\[ W_0^{s,p}(\Omega) := \left\{ u \in W^{s,p}(\RN): u = 0 \textup{ in } \RN \setminus \Omega \right\}.\]
It is a Banach space endowed with the norm
\[ \|u\|_{W^{s,p}_0(\Omega)} := \left( \iint_{D_{\Omega}} \frac{|u(x)-u(y)|^p}{|x-y|^{N+sp}} dx dy \right)^{1/p},\]
where
\[ D_{\Omega} := (\RN \times \RN) \setminus (\mathcal{C}\Omega \times \mathcal{C}\Omega) = (\Omega \times \RN) \cup (\mathcal{C}\Omega \times \Omega).\]
\end{definition}

The space $W_0^{s,p}(\Omega)$ was first introduced in \cite{S_V_2012} in the particular case $p = 2$. We refer to \cite{DN_P_V_2012} for more details on fractional Sobolev spaces. Nevertheless, due to their relevance in this work, we recall here some results involving fractional Sobolev spaces.
\medbreak
We shall make use of the following classical fractional Sobolev inequality. See \cite[Proposition 15.5]{Ponce_Book_2015} for a beautiful proof.

\begin{theorem}[Sobolev inequality] \label{Sobolev inequality}
For any $s \in (0,1)$, $p \in [1, \frac{N}{s})$ and $u \in W^{s,p}(\RN)$, it follows that
\[ \|u\|_{L^{p_s^{\ast}}}(\RN) \leq S_{N,p} \left( \iint_{\R^{2N}} \frac{|u(x)-u(y)|^p}{|x-y|^{N+sp}} dx dy \right)^{1/p},\]
where $S_{N,p} > 0$ is a constant depending only on $N$ and $p$.
\end{theorem}

Next, we present a fractional Hardy inequality and some of its consequences. These results will be crucial to show the optimality of the regularity assumptions of Theorem \ref{main th existence}, namely to prove Theorem \ref{optimality}.

\begin{theorem} \label{hardy_Frank_Seiringer} \rm \cite[Theorem 1.1]{F_S_2008}  \it Let $N \geq 2$, $0 < s < 1$ and $p > 1$. Then, for all $u \in \mathcal{C}_0^{\infty}(\RN)$, it follows that
\begin{equation}\label{hardy}
\iint_{\R^{2N}}
\dfrac{|u(x)-u(y)|^p}{|x-y|^{N+ps}}dxdy\ge \L_{N,p,s}\int_{\RN}
\dfrac{|u(x)|^p}{|x|^{ps}}dx,
\end{equation}
where
\begin{equation} \label{LL}
\L_{N,s,p} := 2 \int_0^1 \sigma^{ps-1} \left| 1 - \sigma^{\frac{N-ps}{p}} \right| \Phi_{N,s,p}(\sigma) d\sigma > 0,
\end{equation}
and
\[ \Phi_{N,s,p}(\sigma) := |\mathbb{S}^{N-2}| \int_{-1}^{1} \frac{(1-t^2)^{\frac{N-3}{2}}}{(1-2\sigma t+\sigma^2)^{\frac{N+ps}{2}}} dt.\]
\end{theorem}

\begin{prop} \label{cor_Hardy} Let $\Omega \subset \RN$, $N \geq 2$, be a bounded domain with boundary $\partial \Omega$ of class $\mathcal{C}^{2}$ such that $0 \in \Omega$, $0 < s < 1$ and $p > 1$. Then:
\begin{itemize}
\item[1)] \rm \cite[Lemma 3.4]{A_B_2017} \it If we set
\[ \L(\O) := \inf \left\{ \dfrac{ \displaystyle \iint_{D_{\Omega}}\dfrac{|\phi(x)-\phi(y)|^p}{|x-y|^{N+ps}}dxdy}{\dyle\io\dfrac{|\phi(x)|^p}{|x|^{ps}}dx} : \phi \in \mathcal{C}_0^{\infty}(\Omega) \setminus \{0\} \right\},\]
it follows that $\L(\Omega) = \L_{N,s,p}$ where $\L_{N,s,p} > 0$ is defined in \eqref{LL}. $ $\\
\item[2)] The weight $|x|^{-ps}$ is optimal in the sense that, for all $\e>0$, if follows that
\[ \inf \left\{ \dfrac{ \displaystyle \iint_{D_{\Omega}}
\dfrac{|\phi(x)-\phi(y)|^p}{|x-y|^{N+ps}}dxdy}{\dyle\io\dfrac{|\phi(x)|^p}{|x|^{ps+\e}}dx } : \phi \in \mathcal{C}_0^{\infty}(\Omega) \setminus \{0\} \right\} = 0.\]
\end{itemize}
\end{prop}

\begin{proof}
Since the proof of 1) can be found in \cite[Lemma 3.4]{A_B_2017}, we just provide the proof of 2). Let $\epsilon > 0$ be fixed but arbitrarily small. We assume by contradiction that there exists a smooth bounded domain $\Omega \subset \RN$ such that $0 \in \Omega$ and
\begin{equation} \label{contr1}
\Lambda_{\epsilon}(\Omega):= \inf \left\{ \dfrac{ \displaystyle \iint_{D_{\Omega}}
\dfrac{|\phi(x)-\phi(y)|^p}{|x-y|^{N+ps}}dxdy}{\dyle\io\dfrac{|\phi(x)|^p}{|x|^{ps+\e}}dx } : \phi \in \mathcal{C}_0^{\infty}(\Omega) \setminus \{0\} \right\} > 0.
\end{equation}
Let us then observe that for any $B_r(0) \subset \Omega$, it follows that
\begin{equation}\label{contr2}
0 < \L_{\epsilon}(\Omega) \leq \L_{\epsilon}(B_r(0)).
\end{equation}
Moreover, observe that for $\phi \in \mathcal{C}_0^{\infty}(B_r(0))$ we have that
\begin{equation} \label{contr3}
\int_{B_r(0)} \frac{|\phi(x)|^p}{|x|^{ps+\epsilon}} dx \geq \frac{1}{r^{\epsilon
}} \int_{B_r(0)} \frac{|\phi(x)|^p}{|x|^{ps}} dx.
\end{equation}
Hence, gathering \eqref{contr2}-\eqref{contr3}, it follows that, for all $\phi \in \mathcal{C}_0^{\infty}(B_r(0))$,
\begin{equation*}
0 < \L_{\epsilon}(\Omega) \leq \Lambda_{\epsilon}(B_r(0)) \leq \dfrac{ \displaystyle \iint_{D_{B_r(0)}}
\dfrac{|\phi(x)-\phi(y)|^p}{|x-y|^{N+ps}}dxdy}{\dyle\int_{B_r(0)}\dfrac{|\phi(x)|^p}{|x|^{ps+\e}}dx } \leq  r^{\epsilon} \dfrac{ \displaystyle \iint_{D_{B_r(0)}}
\dfrac{|\phi(x)-\phi(y)|^p}{|x-y|^{N+ps}}dxdy}{\dyle\int_{B_r(0)}\dfrac{|\phi(x)|^p}{|x|^{ps}}dx }.
\end{equation*}
Thus, by the definition of $\Lambda(B_r(0))$ and 1), we deduce that
$0 < \frac{\Lambda_{\epsilon}(\Omega)}{r^{\epsilon}} \leq \L_{B_r(0)} = \L_{N,s,p}$. Since (by assumption) $\L_{\epsilon}(\Omega) > 0$ and $\L_{N,s,p}$ is independent of $\Omega$, letting $r \to 0$, we obtain a contradiction and the result follows.
\end{proof}
 
In order to prove some of the Calder\'on-Zygmund type regularity results of Section \ref{3}, we will use the relation between the fractional Sobolev space $W^{s,p}(\ren)$ and the Bessel potential space defined below.
\begin{definition} \label{Bessel Space}
Let $s \in (0,1)$. For any $p \in [1,\infty)$, the Bessel potential space $L^{s,p}(\RN)$ is defined as
$$
L^{s,p}(\ren):= \left\{ u \in L^p(\RN) \mbox{  such that   } u = (I-\D)^{-\frac{s}{2}}f \mbox{  with } f\in L^p(\RN)\right\}.$$
It is a Banach space endowed with the norm
$$
\vertiii{u}_{L^{s,p}(\ren)}:= \|u\|_{L^p(\ren)}+ \|f\|_{L^p(\ren)}.
$$
\end{definition}

Having in mind the \textit{fractional gradient of order $s$} introduced in \eqref{frac-g}, let us point out that in \cite[Theorem 1.7]{S_S_2015} it is proved that
\begin{equation}\label{equiv}
L^{s,p}(\ren)=\left\{ u \in L^p(\RN) \mbox{  such that   } |\n^s u|\in L^p(\RN)\right\},
\end{equation}
with the equivalent norm
$$
\|u\|_{L^{s,p}(\ren)} := \|u\|_{L^p(\ren)}+ \|\n^s u\|_{L^p(\ren)}.
$$
\medbreak

Notice also that in the case where $s$ is an integer and $1<p<\infty$, by \cite[Theorem 7.63]{adams} we know that $L^{s,p}(\RN)=W^{s,p}(\RN)$.  Differently, in case $s\in (0,1)$, the two previous spaces does not coincide. However, we have the following result.
\begin{theorem}\label{two-spaces} \rm \cite[Theorem 7.63]{adams} \it Assume that $s\in (0,1)$ and $1<p<\infty$. For all $0 < \e < s$, it follows that
$$
L^{s+\e, p}(\RN)\subset W^{s,p}(\ren)\subset L^{s-\e,p}(\RN),
$$
with continuous inclusions.
\end{theorem}

Finally, we recall a classical result of harmonic analysis that will be useful in Section \ref{3}.

\begin{lemma}\label{Stein}\rm \cite[Theorem I, Section 1.2, Chapter V]{Stein_book_1970} \it
Let $0<\lambda<N$ and  $1\le p<\ell<\infty$ be such that  $\dfrac{1}{\ell}+1=\dfrac{1}{p}+\dfrac{\l}{N}$. For $g\in L^p(\RN)$, we define $$J_\lambda(g)(x)=\int_{\RN}
\dfrac{g(y)}{|x-y|^\l}dy.$$
Then, it follows that:
\begin{itemize}
\item[$a)$] $J_\lambda$ is well defined in the sense that the integral converges absolutely for  almost all $x\in \mathbb{R}^N$. \vspace{0.15cm}
\item[$b)$] If $p>1$, then $\|J_{\lambda}(g)\|_{L^\ell(\RN)} \leq c_{p,q} \|g\|_{L^p(\RN)}.$ 
    \item[$c)$] If $p=1$, then $\big|\{x\in \mathbb{R}^N\,| J_\lambda(g)(x)>\sigma\}\big|\le \left(\dfrac{
    A\|g\|_{L^1(\RN)}}{\sigma}\right)^\ell$.
\end{itemize}
\end{lemma}

\section{Regularity results for the fractional Poisson equation}\label{3}

The main goal of this section,  which is independent of the rest of the work, is to prove sharp Calder\'on-Zygmund type regularity results for the fractional Poisson equation 
\begin{equation} \label{linearEq}
\left\{
\begin{aligned}
\sLap v & = h(x), \quad &  \textup{ in } \Omega,\\
v & = 0, & \textup{ in } \RN \setminus \Omega,
\end{aligned}
\right.
\end{equation}
under the assumption 
\begin{equation} \label{hyplinearEq}
\left\{
\begin{aligned}
& \Omega \subset \RN,\ N \geq 2, \textup{ is a bounded domain with }\partial \Omega \textup{ of class } \mathcal{C}^{2}, \\
& s \in (1/2,1),\\
& h \in L^m(\Omega) \textup{ for some } m \geq 1.
\end{aligned}
\right.
\end{equation}
First of all, let us precise the notion of weak solution to \eqref{linearEq}.
\medbreak
\begin{definition}
We say that $v$ is a\textit{ weak solution} to \eqref{linearEq} if $v \in L^1(\Omega)$, $v \equiv 0$ in $\mathcal{C}\Omega := \RN \setminus \Omega$ and
\[ \int_{\Omega} v \sLap \phi\, dx = \int_{\Omega} h(x) \phi\, dx\,,\quad \forall\  \phi \in \mathbb{X}_s,\]
where $\mathbb{X}_s$ is defined in \eqref{Xs}.
\end{definition}
\medbreak

Under our assumption \eqref{hyplinearEq}, the existence and uniqueness of solutions to \eqref{linearEq} is a particular case of \cite[Proposition 2.4]{C_V_2014-JDE} (see also \cite[Section 4]{L_P_P_S_2015}). Having this in mind, we prove several regularity results for \eqref{linearEq}. Our first main result reads as follows:

\begin{prop} \label{corollary regularity} \label{regularity}
Assume \eqref{hyplinearEq} and let $v$ be the unique weak solution to \eqref{linearEq} and $t \in (0,1)$:
\begin{itemize}
\item[1)] If $m=1$, then $v \in W_0^{t,p}(\Omega)$ for all $1 \leq p < \frac{N}{N-(2s-t)}$ and there exists $C_1 = C_1(s,t,p,\Omega) > 0$ such that 
\[ \|v\|_{W_0^{t,p}(\Omega)}  \leq \|v\|_{W^{t,p}(\RN)}  \leq C_1 \|h\|_{L^1(\Omega)}.\]

\item[2)] If $1 < m < \frac{N}{2s}$, then $v \in W_0^{t,p}(\Omega)$ for all $1 \leq p \leq \frac{mN}{N-m(2s-t)}$ and there exists  $C_1 = C_1(m,s,t,p,\Omega) > 0$ such that
\[ \|v\|_{W_0^{t,p}(\Omega)}  \leq \|v\|_{W^{t,p}(\RN)}  \leq C_1 \|h\|_{L^m(\Omega)}.\]
\item[3)] If $\frac{N}{2s} \leq m < \frac{N}{2s-1}$, then $v \in W_0^{t,p}(\Omega)$ for all $1 \leq p < \frac{mN}{t(N-m(2s-1))}$  and there exists $C_1 = C_1(m,s,t,p,\Omega) > 0$ such that
\[ \|v\|_{W_0^{s,p}(\Omega)}  \leq \|v\|_{W^{t,p}(\RN)} \leq C_1 \|h\|_{L^m(\Omega)}.\]
\item[4)] If $m \geq \frac{N}{2s-1}$, then $v \in W_0^{t,p}(\Omega)$ for all $1 \leq p < \infty$.
\end{itemize}
\end{prop}

\begin{remark} $ $
\begin{itemize}
\item[a)] The previous results are sharp in the sense that, if ``we take $t = s = 1$'', we recover the classical sharp regularity results for the local case and those cannot be improved. See for instance \cite[Chapter 5]{Ponce_Book_2015}.
\item[b)] In the particular case of the fractional Laplacian of order $s \in (1/2,1)$ and for $h \in L^1(\Omega)$, we improve the regularity results of \cite{A_A_B_2016, K_M_S_2015, L_P_P_S_2015}. Note however that in the three quoted papers the authors deal with more general operators and cover the full range $s \in (0,1)$. Furthermore, in \cite{K_M_S_2015} the authors also deal with measures as data.
\item[c)] Since $s \in (1/2,1)$, observe that $t < 2s$ for all $t \in (0,1)$.\end{itemize}
\end{remark}

As we believe it has its own interest, let us highlight a particular case of the previous result which follows directly from Proposition \ref{corollary regularity} considering $t = s$.

\begin{cor} 
Assume \eqref{hyplinearEq} and let $v$ be the unique weak solution to \eqref{linearEq}:
\begin{itemize}
\item[1)] If $m=1$, then $v \in W_0^{s,p}(\Omega)$ for all $1 \leq p < \frac{N}{N-s}$ and there exists $C_1 = C_1(s,p,\Omega) > 0$ such that
\[ \|v\|_{W_0^{s,p}(\Omega)} \leq \|v\|_{W^{s,p}(\RN)} \leq C_1 \|h\|_{L^1(\Omega)}.\]

\item[2)] If $1 < m < \frac{N}{2s}$, then $v \in W_0^{s,p}(\Omega)$ for all $1 \leq p \leq \frac{mN}{N-ms}$ and there exists $C_1 = C_1(m,s,p,\Omega) > 0$ such that
\[ \|v\|_{W_0^{s,p}(\Omega)} \leq \|v\|_{W^{s,p}(\RN)}  \leq C_1 \|h\|_{L^m(\Omega)}.\]
\item[3)] If $\frac{N}{2s} \leq m < \frac{N}{2s-1}$, then $v \in W_0^{s,p}(\Omega)$ for all $1 \leq p < \frac{mN}{s(N-m(2s-1))}$  and there exists $C_1 = C_1(m,s,p,\Omega) > 0$ such that
\[ \|v\|_{W_0^{s,p}(\Omega)}  \leq \|v\|_{W^{s,p}(\RN)} \leq C_1 \|h\|_{L^m(\Omega)}.\]
\item[4)] If $ m \geq \frac{N}{2s-1}$, then $v \in W_0^{s,p}(\Omega)$ for all $1 \leq p < \infty$.
\end{itemize}
\end{cor}

In the following two results we complete the information obtained in Proposition \ref{regularity} when $h \in L^m(\Omega)$ for some $m > N/2s$.

\begin{prop} \label{corollary regularity-2}
Assume \eqref{hyplinearEq} and let $v$ be the unique weak solution to \eqref{linearEq} and $t \in (0,s)$:  
\begin{itemize}
\item[1)] If $\frac{N}{2s} \leq m < \frac{N}{2s-t}$ then $v \in W_0^{t,p}(\Omega)$ for all $1 \leq p < \frac{mN}{N-m(2s-t)}$ and there exists  $C_2 = C_2(m,s,t,p,\Omega) > 0$ such that
\[ \|v\|_{W_0^{t,p}(\Omega)}  \leq \|v\|_{W^{t,p}(\RN)}  \leq C_2 \|h\|_{L^m(\Omega)}.\]
\item[2)] If $m \geq \frac{N}{2s-t}$ then $v \in W_0^{t,p}(\Omega)$ for all $1 \leq p < \infty$ and there exists $C_2 = C_2(m,s,t,p,\Omega) > 0$ such that
\[ \|v\|_{W_0^{t,p}(\Omega)}  \leq \|v\|_{W^{t,p}(\RN)}  \leq C_2 \|h\|_{L^m(\Omega)}.\]
\end{itemize}
\end{prop}

\begin{prop} \label{corollary regularity-3}
Assume \eqref{hyplinearEq} and let $v$ be the unique weak solution to \eqref{linearEq} and $t \in (s,1)$. If $\frac{N}{2s} \leq m < \frac{N}{s}$ then $v \in W_0^{t,p}(\Omega)$ for all $1 \leq p < \frac{mN}{N-m(2s-t)}$ and there exists  $C_2 = C_2(m,s,t,p,\Omega) > 0$ such that
\[ \|v\|_{W_0^{t,p}(\Omega)}  \leq \|v\|_{W^{t,p}(\RN)}  \leq C_2 \|h\|_{L^m(\Omega)}.\]
\end{prop}

\begin{remark}
Notice that, in the case where $t\in (s,1)$, Propositions \ref{corollary regularity} and \ref{corollary regularity-3} complete and somehow give a more precise information than the result obtained in \cite{K_M_S_2015-2}. 
\end{remark}

\begin{remark}
The proofs of Propositions \ref{regularity}, \ref{corollary regularity-2} and \ref{corollary regularity-3} are postponed to Subsection \ref{3.1}
\end{remark}

Due to the nonlocality of the fractional Laplacian, several notions of regularity can be studied. The following results, which generalize the fractional regularity proved in \cite[Theorem 24]{L_P_P_S_2015} with a different approach, can be seen as the counterpart of Proposition \ref{regularity} to deal with \eqref{Qlambda} and \eqref{Rlambda}.

\begin{prop} \label{regu-grad1}
Assume \eqref{hyplinearEq} and let $v$ be the unique weak solution to \eqref{linearEq} and $t \in (0,s]$:
\begin{itemize}
\item[1)] If $m = 1$, then $(-\Delta)^{\frac{t}{2}}v \in L^p(\Omega)$ for all $1 \leq p < \frac{N}{N-(2s-t)}$ and there exists $C_3 = C_3(s,t,p,\Omega) > 0$ such that
\[ \|(-\Delta)^{\frac{t}{2}}v\|_{L^p(\O)}\leq C_3 \|h\|_{L^1(\Omega)}.\]
\item[2)] If $1 < m < \frac{N}{2s-t}$, then $(-\Delta)^{\frac{t}{2}}v \in L^p(\Omega)$ for all $1 \leq p \leq \frac{mN}{N-m(2s-t)}$ and there exists $C_3 = C_3(s,t,m,p,\Omega) > 0$ such that
\[ \|(-\Delta)^{\frac{t}{2}}v\|_{L^p(\O)}\leq C_3 \|h\|_{L^m(\Omega)}.\]
\item[3)] If $m \geq \frac{N}{2s-t}$ then  $(-\Delta)^{\frac{t}{2}}v \in L^p(\Omega)$ for all $1 \leq p < \infty$ and there exists $C_3 = C_3(s,t,m,p,\Omega) > 0$ such that
\[ \|(-\Delta)^{\frac{t}{2}}v\|_{L^p(\O)}\leq C_3 \|h\|_{L^m(\Omega)}.\]
\end{itemize}
\end{prop}

\begin{cor} \label{corollary regu-grad1}
Assume \eqref{hyplinearEq} and let $v$ be the unique weak solution to \eqref{linearEq}:
\begin{itemize}
\item[1)] If $m = 1$, then $|\nabla^s v| \in L^p(\Omega)$ for all $1 \leq p < \frac{N}{N-s}$ and there exists $C_4 = C_4(s,p,\Omega) > 0$ such that
\[ \|\nabla^s v\|_{L^p(\O)}\leq C_4 \|h\|_{L^1(\Omega)}.\]
\item[2)] If $1 < m < \frac{N}{s}$, then $|\nabla^s v| \in L^p(\Omega)$ for all $1 \leq p \leq \frac{mN}{N-ms}$ and there exists $C_4 = C_4(s,m,p,\Omega) > 0$ such that
\[ \|\nabla^s v\|_{L^p(\O)}\leq C_4 \|h\|_{L^m(\Omega)}.\]
\item[3)] If $m \geq \frac{N}{s}$ then  $|\nabla^s v| \in L^p(\Omega)$ for all $1 \leq p < \infty$ and there exists $C_4 = C_4(s,m,p,\Omega) > 0$ such that
\[ \|\nabla^s v\|_{L^p(\O)}\leq C_4 \|h\|_{L^m(\Omega)}.\]
\end{itemize}
\end{cor}

\begin{remark}
The proofs of Proposition \ref{regu-grad1} and Corollary \ref{corollary regu-grad1} will be given in Subsection \ref{3.2}
\end{remark}

Now, before proving Propositions \ref{regularity}, \ref{corollary regularity-2}, \ref{corollary regularity-3}, \ref{regu-grad1} and Corollary \ref{corollary regu-grad1}, we state some known results that will be useful in our proofs. First of all, we gather in the following lemma several results of \cite{L_P_P_S_2015}. See also \cite[Section 2]{C_V_2014-JDE} and \cite{B_W_Z_2017}.

\begin{lemma} \label{LPPS weak solution regularity}  \rm \cite[Theorems 13, 15, 16, 23, 24]{L_P_P_S_2015} \it Let $\Omega \subset \RN$, $N \geq 2$, be a bounded domain with boundary $\partial \Omega$ of class $\mathcal{C}^{0,1}$, let $s \in (0,1)$ and assume that $h \in L^m(\Omega)$ for some $m \geq 1$. Then problem \eqref{linearEq} has an unique weak solution. Moreover:
\begin{itemize}
\item[1)] If $m=1$, then $v \in L^p(\Omega)$ for all $1 \leq p < \frac{N}{N-2s}$ and there exists $C_5 = C_5(s,p,\Omega) > 0$ such that
\[ \|v\|_{L^p(\O)}  \leq C_5 \|h\|_{L^1(\Omega)}.\]
\item[2)] If $1 < m < \frac{N}{2s}$, then $v \in L^p(\Omega)$ for all $1 \leq p \leq \frac{mN}{N-2ms}$ and there exists $C_5 = C_5(s,m,p,\Omega) > 0$ such that
\[ \|v\|_{L^p(\Omega)} \leq C_5 \|h\|_{L^m(\Omega)}.\]
\item[3)] If $m  \geq \frac{N}{2s}$, then $v \in L^p(\Omega)$ for all $1 \leq p < \infty$ and there exists $C_5 = C_5(s,m,p,\Omega) > 0$ such that
\[ \|v\|_{L^p(\Omega)} \leq C_5 \|h\|_{L^m(\Omega)}.\]
\end{itemize}
\end{lemma}

\bigbreak
Considering stronger assumptions, the first author and I. Peral proved in \cite{A_P_2018} that the unique weak solution to \eqref{linearEq} belongs to a suitable local Sobolev space. More precisely, under the assumption \eqref{hyplinearEq} the authors obtained the following result.
\begin{lemma} \label{AP gradient regularity} \rm \cite[Lemma 2.15]{A_P_2018} \it
Assume \eqref{hyplinearEq} and let $v$ be the unique weak solution to \eqref{linearEq}:
\begin{itemize}
\item[1)] If $m=1$, then $v \in W^{1,p}(\RN)$ for all $1 \leq p < \frac{N}{N-(2s-1)}$ and there exists $C_6 = C_6(s,p,\Omega) > 0$ such that
\[ \|v\|_{W^{1,p}(\RN)} \leq \overline{C}_6 \|\gradv\|_{L^p(\Omega)}  \leq C_6 \|h\|_{L^1(\Omega)}.\]
\item[2)] If $1 < m < \frac{N}{2s-1}$, then $v \in W^{1,p}(\RN)$ for all $ 1 \leq p \leq \frac{mN}{N-m(2s-1)}$ and there exists $C_6 = C_6(m,s,p,\Omega) > 0$ such that
\[ \|v\|_{W^{1,p}(\RN)} \leq \overline{C}_6 \|\gradv\|_{L^p(\Omega)}  \leq C_6 \|h\|_{L^m(\Omega)}.\]
\item[3)] If $m\geq \frac{N}{2s-1}$, then $v \in W^{1,p}(\RN)$ for all $1 \leq p < \infty$ and there exists $C_6 = C_6(m,s,p,\Omega)> 0$ such that
\[ \|v\|_{W^{1,p}(\RN)} \leq \overline{C}_6 \|\gradv\|_{L^p(\Omega)}  \leq C_6 \|h\|_{L^m(\Omega)}.\]
\end{itemize}
\end{lemma}

\medbreak
As last ingredient to prove our regularity results we need an interpolation result that we borrow from \cite{B_M_2018}. Let us introduce the real numbers $0 \leq s_1 \leq \eta \leq s_2 \leq 1$ and $1 \leq p_1,p_2,p \leq \infty$ and assume that  they satisfy the relations
\begin{equation} \label{BM 2018 1.3}
 \eta = \theta s_1 + (1-\theta)s_2 \quad \textup{ and } \quad \frac{1}{p} = \frac{\theta}{p_1} + \frac{1-\theta}{p_2} \quad \textup{with} \quad 0 < \theta < 1.
\end{equation}
Moreover, let us introduce the condition
\begin{equation} \label{BM 2018 1.9}
s_2 = p_2 = 1 \quad \textup{ and } \quad \frac{1}{p_1} \leq s_1.
\end{equation}

\medbreak

\begin{lemma}\label{BM interpolation} \rm \cite[Theorem 1]{B_M_2018} \it
Assume that \eqref{BM 2018 1.3} holds and \eqref{BM 2018 1.9} fails. Then, for every $\theta \in (0,1)$, there exists a constant $C = C(s_1,s_2,p_1,p_2,\theta) > 0$ such that
\[ \|w\|_{W^{\eta,p}(\RN)} \leq C \|w\|_{W^{s_1,p_1}(\RN)}^{\theta} \|w\|_{W^{s_2,p_2}(\RN)}^{1-\theta}\,, \quad  \forall\ w \in W^{s_1,p_1}(\RN) \cap W^{s_2,p_2}(\RN)\,.\]
\end{lemma}

\medbreak

\subsection{Proofs of Propositions \ref{regularity}, \ref{corollary regularity-2} and \ref{corollary regularity-3}.}\label{3.1} $ $
\medbreak

Having at hand all the needed ingredients, we prove our first regularity result.

\begin{proof} [\textbf{Proof of Proposition \ref{regularity}}]
1) Let $v$ be the unique weak solution to \eqref{linearEq}. On the one hand, by Lemma \ref{LPPS weak solution regularity}, 1), we know that
\begin{equation} \label{ineq1A}
\|v\|_{L^{p_1}(\RN)} \leq C_5 \|h\|_{L^1(\Omega)}\,, \qquad \forall\ 1 \leq p_1 < \frac{N}{N-2s}.
\end{equation}
On the other hand, by Lemma \ref{AP gradient regularity}, 1), we know that
\begin{equation} \label{ineq2A}
\|v\|_{W^{1,p_2}(\RN)} \leq C_6 \|h\|_{L^1(\Omega)}\,, \qquad \forall\ 1 \leq p_2 < \frac{N}{N-(2s-1)}.
\end{equation}
Also, by Lemma \ref{BM interpolation} applied with $\eta = t$, $s_1 = 0$ and $s_2 = 1$, we have that
\begin{equation} \label{ineq3A}
\|v\|_{W^{t,p}(\RN)} \leq C \|v\|_{L^{p_1}(\RN)}^{1-t} \|v\|_{W^{1,p_2}(\RN)}^t.
\end{equation}
The result follows from \eqref{ineq1A}-\eqref{ineq3A} using that
\[ 1 \geq \frac{1}{p} =  \frac{1-t}{p_1}+\frac{t}{p_2} > \frac{(1-t)(N-2s)+ t(N-(2s-1))}{N} = \frac{N-(2s-t)}{N}. \]
\medbreak
\noindent 2) Let $v$ be the unique weak solution to \eqref{linearEq}. By Lemma \ref{LPPS weak solution regularity}, 2), we know that
\begin{equation} \label{ineq1A-2}
\|v\|_{L^{p_1}(\RN)} \leq C_5 \|h\|_{L^m(\Omega)}\,, \qquad \forall\ 1 \leq p_1 \leq \frac{mN}{N-2ms}.
\end{equation}
Also, by Lemma \ref{AP gradient regularity}, 2), we know that
\begin{equation} \label{ineq2A-2}
\|v\|_{W^{1,p_2}(\RN)} \leq C_6 \|h\|_{L^m(\Omega)}\,, \qquad \forall\ 1 \leq p_2 \leq \frac{mN}{N-m(2s-1)}.
\end{equation}
Finally, by Lemma \ref{BM interpolation} applied with $\eta = t$, $s_1 = 0$ and $s_2 = 1$, we know that \eqref{ineq3A} holds.
The result follows from  \eqref{ineq3A}, \eqref{ineq1A-2} and \eqref{ineq2A-2} using that
\[ 1 \geq \frac{1}{p} =  \frac{1-t}{p_1}+\frac{t}{p_2} \geq \frac{(1-t)(N-2ms)+ t(N-m(2s-1))}{mN} = \frac{N-m(2s-t)}{mN}. \]
\medbreak
\noindent 3) Let $v$ be the unique weak solution to \eqref{linearEq}. By Lemma \ref{LPPS weak solution regularity}, 3), we know that
\begin{equation} \label{ineq1B}
\|v\|_{L^{p_1}(\RN)}  \leq C_5 \|h\|_{L^m(\Omega)}\,, \qquad \forall\ 1 \leq p_1 < \infty.
\end{equation}
By Lemma \ref{AP gradient regularity}, 2), we know that \eqref{ineq2A-2} holds. Moreover, by Lemma \ref{BM interpolation} applied with $\eta = t$, $s_1 = 0$ and $s_2 = 1$, it follows that \eqref{ineq3A} holds. The result follows from \eqref{ineq3A}, \eqref{ineq2A-2} and  \eqref{ineq1B} using that
\[ 1 \geq \frac{1}{p} = \frac{1-t}{p_1}+\frac{t}{p_2} > \frac{t(N-m(2s-1))}{mN}.\]
\medbreak
\noindent 4) Let $v$ be the unique weak solution to \eqref{linearEq}. By Lemma \ref{LPPS weak solution regularity}, 3), we know that \eqref{ineq1B} holds. On the other hand, by Lemma \ref{AP gradient regularity}, 3), we have that
\begin{equation} \label{ineq1C}
\|v\|_{W^{1,p_2}(\RN)} \leq C_6 \|h\|_{L^m(\Omega)}\,, \qquad \forall\ 1 \leq p_2 < \infty.
\end{equation}
The result follows from Lemma \ref{BM interpolation} applied with $\eta = t, s_1 = 0$ and $s_2 = 1$.
\end{proof}

\begin{remark}
The restriction $s \in (1/2,1)$ comes from Lemma \ref{AP gradient regularity}. It is not expected that Lemma \ref{AP gradient regularity} holds true for $s \in (0,1/2]$. Hence, with this approach we are limited to deal with $s \in (1/2,1)$.
\end{remark}

\medbreak
Now, using the Bessel potential space $L^{s,p}(\RN)$ (see Definition \ref{Bessel Space}), Theorem \ref{two-spaces} and Lemma \ref{adams-interpolation} below, we prove Propositions \ref{corollary regularity-2} and \ref{corollary regularity-3}. We begin proving a regularity result in the Bessel potential space.

\begin{lemma}  \label{adams-interpolation} \rm \cite[Theorem 7.58]{adams} \it
Let $0<t<s<1$, $1 < p < \frac{N}{s-t}$ and $q = \frac{Np}{N-p(s-t)}$. Then $W^{s,p}(\RN) \subset W^{t,q}(\RN)$ with continuous inclusion.
\end{lemma}

\begin{prop} \label{regularity-bessel}
Assume \eqref{hyplinearEq} and let $v$ be the unique weak solution to \eqref{linearEq}:
\begin{itemize}
\item[1)] If $m=1$, then $v \in L^{s,p}(\ren)$ for all $1 \leq p < \frac{N}{N-s}$ and there exists $C_7 = C_7(s,p,\Omega) > 0$ such that
\[ \|v\|_{L^{s,p}(\RN)} \leq C_7 \|h\|_{L^1(\Omega)}.\]
\item[2)] If $1<m<\frac{N}{s}$, then $v \in L^{s,p}(\RN)$ for all $ 1 \leq p \leq \frac{mN}{N-ms}$ and there exists $C_7 = C_7(m,s,p,\Omega) > 0$ such that
\[ \|v\|_{L^{s,p}(\RN)} \leq C_7 \|h\|_{L^m(\Omega)}.\]
\item[3)] If $m  \geq \frac{N}{s}$, then $v \in L^{s,p}(\RN)$ for all $1 \leq p < \infty$ and there exists $C_7 = C_7(m,s,p,\Omega) > 0$ such that
\[ \|v\|_{L^{s,p}(\RN)} \leq C_7 \|h\|_{L^m(\Omega)}.\]
\end{itemize}
\end{prop}

\begin{proof}
Assume \eqref{hyplinearEq} and let $v$ be the unique weak solution to \eqref{linearEq}. Taking into account \eqref{equiv}, we have just to show the regularity of $|\n^sv|$ where $\n^s$ is defined in \eqref{frac-g}.
By a density argument and \cite[Lemma 15.9]{Ponce_Book_2015} we have that
\begin{equation}\label{control01}
|\n^s v(x)|\le \frac{1}{N-(1-s)}\irn \frac{|\n v(y)|}{|x-y|^{N-(1-s)}}dy\;\, \quad \textup{ a.e. in }\ \RN.
\end{equation}
Let us then split into three cases:
\medbreak
\noindent 1) $m = 1$.
\medbreak
By Lemma \ref{AP gradient regularity}, 1), we get $v \in W^{1,q}(\RN)$ for all $1 \leq q < \frac{N}{N-(2s-1)}$. Thus, by Lemma \ref{Stein}, we conclude that $|\n^s v(x)|\in L^p(\RN)$ for all $ 1 \leq p < \frac{N}{N-s}$.
\medbreak
\noindent 2) $1 < m < \frac{N}{s}$.
\medbreak
The result follows arguing on the same way, using now \eqref{control01}, Lemma \ref{AP gradient regularity}, 2) and Lemma \ref{Stein}.
\medbreak
\noindent 3) $ m \geq \frac{N}{s}$.
\medbreak
In this case, since $m \geq \frac{N}{s}$ and $\Omega$ is a bounded domain then $f\in L^{\bar{m}}(\O)$ for all $\bar{m} < \frac{N}{s}$. In particular, using the second point it follows that $v\in L^{s,\bar{p}}(\RN)$ for all $1 \leq \bar{p} \leq \frac{\bar{m}N}{N-s\bar{m}}$. Letting $\bar{m}\uparrow \frac{N}{s}$, we reach that $\bar{p}\uparrow \infty$. Hence $v \in L^{s,p}(\RN)$ for all $1 \leq p < \infty$. 
\end{proof}
Now, using the above regularity result in the Bessel potential space, we obtain the following lemma.
\begin{lemma} \label{regularity-2}
Assume \eqref{hyplinearEq} and let $v$ be the unique weak solution to \eqref{linearEq}:
\begin{itemize}
\item[1)] If $\frac{N}{2s}\le m<\frac{N}{s}$ then $v \in W_0^{s',p}(\Omega)$ for all $0 < s' <s$ and all $1 \leq  p \leq \frac{mN}{N-ms}$. Moreover, there exists $C_8 = C_8(m,s',p,\Omega) > 0$ such that
\[ \|v\|_{W_0^{s',p}(\Omega)} \leq \|v\|_{W^{s',p}(\RN)}  \leq C_8 \|h\|_{L^m(\Omega)}.\]
\item[2)] If $m \geq \frac{N}{s}$, then $v \in W_0^{s',p}(\Omega)$ for all $ 0 < s' < s$ and all $1 \leq p < \infty$. Moreover, there exists $C_8 = C_8(m,s',p,\Omega) > 0$ such that
\[ \|v\|_{W_0^{s',p}(\Omega)}  \leq \|v\|_{W^{s',p}(\RN)}  \leq C_8 \|h\|_{L^m(\Omega)}.\]
\end{itemize}
\end{lemma}

\begin{proof}
First observe that without loss of generality we can assume that $p > 1$. For $p = 1$ the result follows from Proposition \ref{regularity} and the continuous embedding $W^{s,p}(\RN) \subset W^{s',p}(\RN)$. We consider then two cases:
\medbreak
\noindent 1)  $\frac{N}{2s}<m<\frac{N}{s}$
\medbreak
Let $v$ be the unique weak solution to \eqref{linearEq}. By Proposition \ref{regularity-bessel}, 2) we know that $v\in L^{s,p}(\RN)$ for all $1 \leq p \leq \frac{mN}{N-ms}$. Thus, by Theorem \ref{two-spaces} we conclude that
$v \in W_0^{s',p}(\Omega)$ for all $0 < s'<s$ and
\[ \|v\|_{W_0^{s',p}(\Omega)} \leq C\|v\|_{L^{s,p}(\RN)}\le C_8 \|h\|_{L^m(\Omega)}.\]
\bigbreak
\noindent 2) $m \geq \frac{N}{s}$
\medbreak
Let $v$ be the unique weak solution to \eqref{linearEq}. In this case, by Proposition \ref{regularity-bessel},  we know that  $v\in L^{s,p}(\RN)$  for all $1 \leq p < \infty$. Hence, by Theorem \ref{two-spaces}, we conclude.
\end{proof}

Having at hand Lemmas  \ref{adams-interpolation} and \ref{regularity-2}, we prove Propositions \ref{corollary regularity-2} and \ref{corollary regularity-3}.

\begin{proof}[\textbf{Proof of Proposition \ref{corollary regularity-2}}]
First observe that, since $t \in (0,s)$ it follows that $\frac{N}{2s-t} < \frac{N}{s}$. Then we split into two cases:
\medbreak
\noindent 1) $\frac{N}{2s} \leq m < \frac{N}{2s-t}$.
\medbreak
Let $v$ be the unique weak solution to \eqref{linearEq}. On the one hand, by Lemma \ref{regularity-2}, 1), we have that $v \in W^{s',p_1}(\RN)$ for all $0 < s' < s$ and all $1 \leq p_1 \leq \frac{mN}{N-ms}$ and that there exists $C_8 > 0$ such that
\begin{equation}
\|v\|_{W_0^{s',p_1}(\Omega)} \leq \|v\|_{W^{s',p_1}(\RN)}  \leq C_8 \|h\|_{L^m(\Omega)}.
\end{equation}
On the other hand, by Lemma \ref{adams-interpolation}, we have that $v \in W^{\eta,q_1}(\RN)$ for all $0 < \eta < s'<s$ and all $1 \leq q_1 \leq \frac{N p_1}{N-p_1(s'-\eta)} \leq \frac{mN}{N-m(s+s'-\eta)}$. Moreover, there exists $C > 0$ such that
\begin{equation}\label{cor-regu2-eq1}
\|v\|_{W^{\eta,q_1}(\RN)} \leq C \|v\|_{W^{s',p_1}(\RN)} \leq C\,C_8 \|h\|_{L^m(\Omega)}.
\end{equation}
We fix then $p \in [1, \frac{mN}{N-m(2s-t)})$ and observe that we can find $\overline{\eta} \in (t,s')$ such that
\[ 1 \leq p \leq \frac{mN}{N-m(s+s'-\overline{\eta})}.\]
The result follows from \eqref{cor-regu2-eq1} using the continuous embedding $W^{\overline{\eta},p}(\RN) \subset W^{t,p}(\RN)$.
\medbreak
\noindent 2) $m \geq \frac{N}{2s-t}$.
\medbreak
The result follows arguing as in the proof of Proposition \ref{regularity-bessel}, 3) using Proposition \ref{corollary regularity-2}, 1).
\end{proof}

\begin{proof} [\textbf{Proof of Proposition \ref{corollary regularity-3}}]
On the one hand, by Lemma \ref{regularity-2}, 1), we have that
\begin{equation} \label{c310-1}
\|v\|_{W_0^{s',p_1}(\Omega)} \leq \|v\|_{W^{s',p_1}(\RN)}  \leq C_8 \|h\|_{L^m(\Omega)}, \quad \forall\ 0 < s'<s,\ \forall\  1 \leq p_1 \leq \frac{mN}{N-ms}.
\end{equation}
On the other hand, by Lemma \ref{AP gradient regularity}, 2), it follows that
\begin{equation} \label{c310-2}
\|v\|_{W^{1,p_2}(\RN)} \leq C_6 \|h\|_{L^m(\Omega)}, \quad \forall\ 1 \leq p_2 \leq \frac{mN}{N-m(2s-1)}.
\end{equation}
Also, by Lemma \ref{BM interpolation} applied with $\eta = t$, $s_1 = s'$ and $s_2 = 1$, we know that
\begin{equation} \label{c310-3}
\|v\|_{W^{t,p'}(\RN)} \leq C \|v\|_{W^{s',p_1}(\RN)}^{\frac{1-t}{1-s'}} \|v\|_{W^{1,p_2}(\RN)}^{\frac{t-s'}{1-s'}},
\end{equation}
with
\[ \frac{1}{p'} = \frac{1}{1-s'} \left( \frac{1-t}{p_1} + \frac{t-s'}{p_2} \right).\]
We fix then an arbitrary $1 \leq p < \frac{mN}{N-m(2s-t)}$ and observe that we can choose $s' < s$ such that $p' = p$. Hence, the result follows from \eqref{c310-1}-\eqref{c310-3}.
\end{proof}

\subsection{Proofs of Proposition \ref{regu-grad1} and Corollary \ref{corollary regu-grad1}} \label{3.2} $ $

\medbreak

Next, using again Lemma \ref{AP gradient regularity} but with a different approach, we prove Proposition \ref{regu-grad1}. As a consequence we will obtain Corollary \ref{corollary regu-grad1}.

\begin{proof}[\textbf{Proof of Proposition \ref{regu-grad1}}]
Let $v$ be the unique weak solution to \eqref{linearEq} and define,  for $x \in \RN$ arbitrary,
\begin{equation*}
\begin{aligned}
S_1:= \{y \in \RN:\, \dist(y,\Omega) > 2 \} \quad \textup{ and } \quad
S_2:= \{y \in \RN:\, \dist(y,\Omega) \leq 2 \textup{ and } |x-y| \geq  1\}.
\end{aligned}
\end{equation*}
Then, observe that, for all $x \in \Omega$,
\begin{equation} \label{I0}
\begin{aligned}
|(-\Delta)^{\frac{t}{2}}v(x)| & \leq \, \int_{\RN} \frac{|v(x)-v(y)|}{|x-y|^{N+t}} dy \\
& \leq \int_{S_1}  \frac{|v(x)-v(y)|}{|x-y|^{N+t}} dy + \int_{S_2}  \frac{|v(x)-v(y)|}{|x-y|^{N+t}} dy + \int_{B_1(x)}  \frac{|v(x)-v(y)|}{|x-y|^{N+t}} dy \\
& =: I_1(x) + I_2(x) + I_3(x).
\end{aligned}
\end{equation}
Now, let us estimate each one of the three terms. First observe that
\begin{equation} \label{I1}
\begin{aligned}
I_1(x) & = \int_{S_1} \frac{|v(x)|}{|x-y|^{N+t}} dy \leq \int_{S_1} \frac{|v(x)|}{\dist(y,\Omega)^{N+t}} dy \\
& = |v(x)| \int_{S_1} \frac{dy}{\dist(y,\Omega)^{N+t}} dy = c_1(N,t,\Omega)|v(x)|, \qquad \forall\ x \in \Omega.
\end{aligned}
\end{equation}
Next, using that $\Omega$ is a bounded domain and the triangular inequality, we deduce that
\begin{equation}\label{I2}
I_2(x) \leq \int_{S_2} |v(x)-v(y)| dy \leq c_2(\Omega)|v(x)| + \|v\|_{L^1(\Omega)}, \qquad \forall\ x \in \Omega.
\end{equation} 
Finally, following the arguments of \cite[Proposition 2.2]{DN_P_V_2012}, we deduce that
\begin{equation} \label{I3}
\begin{aligned}
I_3(x) & = \int_{B_1(0)} \frac{|v(x)-v(x+z)|}{|z|} \frac{1}{|z|^{N+t-1}}dz = \int_{B_1(0)} \int_0^1 \frac{|\gradv (x+\tau z)|}{|z|^{N+t-1}} d\tau dz \\
& \leq  \int_0^1 \int_{\RN} \frac{|\nabla v(w)|}{|w-x|^{N+t-1}} \tau^{t-1} dw d\tau = \left( \int_{\RN} \frac{|\nabla v(w)|}{|w-x|^{N+t-1}} dw \right) \left( \int_0^1 \tau^{t-1} d\tau \right) \\
& = \frac{1}{t} \int_{\RN}  \frac{|\nabla v(w)|}{|w-x|^{N+t-1}} dw, \qquad \forall\ x \in \Omega.
\end{aligned}
\end{equation}
From \eqref{I0}-\eqref{I3}, we deduce that
\begin{equation} \label{I4}
|(-\Delta)^{\frac{t}{2}}v(x)|\le c(s,t,\O)\bigg(|v(x)|+\int_{\RN}\frac{|\n v(w)|}{|w-x|^{N+t-1}} dw+\|v\|_{L^1(\O)}\bigg), \qquad \forall\ x \in \Omega,
\end{equation}
and so, exploiting again the fact that $\Omega$ is a bounded domain and using the H\"older and triangular inequalities, for all $1 \leq p < \infty$, we obtain that
\begin{equation}\label{I5}
\|(-\Delta)^{\frac{t}{2}} v \|_{L^p(\Omega)} \leq c_2(s,t,\Omega) \left( \|v\|_{L^p(\Omega)} + \left\| \int_{\RN}\frac{|\n v(w)|}{|w-x|^{N+t-1}} dw \right\|_{L^p(\RN)} \right).
\end{equation}
Now, let us split the proof into three parts:
\medbreak
\noindent 1) $m = 1$.
\medbreak
By Lemma \ref{AP gradient regularity}, 1), we know that
$v \in W_0^{1,\sigma}(\Omega)$ for all $1 \leq \sigma < \frac{N}{N-(2s-1)}$ and there exists $C_6 = C_6(s,\sigma,\Omega) > 0$ such that
\[ \|\gradv\|_{L^{\sigma}(\RN)} \leq C_6 \|h\|_{L^1(\Omega)}, \qquad \forall\ 1 \leq \sigma < \frac{N}{N-(2s-1)}.\]
Thus, applying Lemma \ref{Stein}, we deduce that
\begin{equation}\label{I6}
\left\|  \int_{\RN}  \frac{|\nabla v(w)|}{|w-x|^{N+t-1}} dw \right\|_{L^{\ell}(\RN)} \leq C\,C_6 \|h\|_{L^1(\Omega)}, \quad \forall\ 1 \leq \ell < \frac{N}{N-(2s-t)}.
\end{equation}
Also, by Lemma \ref{LPPS weak solution regularity}, 1), we know that
\begin{equation}\label{I7}
\|v\|_{L^{\gamma}(\Omega)} \leq C_5 \|h\|_{L^{1}(\Omega)}, \quad \forall\ 1 \leq \gamma < \frac{N}{N-2s}.
\end{equation}
Taking into account \eqref{I6}-\eqref{I7}, the result follows from \eqref{I5}.
\medbreak
\noindent 2) $1 < m < \frac{N}{2s-t}$.
\medbreak
Observe that, by Lemma \ref{AP gradient regularity}, 2), it follows that $v \in W_0^{1,\sigma}(\Omega)$ for all $1 \leq \sigma \leq \frac{mN}{N-m(2s-1)}$ and there exists $C_6 = C_6(m,s,\sigma,\Omega) > 0$ such that
\[ \|\gradv\|_{L^{\sigma}(\RN)} \leq C_6 \|h\|_{L^m(\Omega)}, \qquad \forall\ 1 \leq \sigma \leq \frac{mN}{N-m(2s-1)}.\]
Hence, by Lemma \ref{Stein}, we deduce that
\begin{equation}\label{I6-2}
\left\|  \int_{\RN}  \frac{|\nabla v(w)|}{|w-x|^{N+t-1}} dw \right\|_{L^{\ell}(\RN)} \leq C\,C_6 \|h\|_{L^m(\Omega)}, \quad \forall\ 1 \leq \ell \leq \frac{mN}{N-m(2s-t)}.
\end{equation}
Also, by Lemma \ref{LPPS weak solution regularity}, 2) and 3) we know that
\begin{equation}\label{I7-2}
\|v\|_{L^{\gamma}(\Omega)} \leq C_5 \|h\|_{L^{m}(\Omega)}, \quad \forall\ 1 \leq \gamma < \frac{mN}{N-2ms}, \quad \textup{ if } \quad 1 \leq m < \frac{N}{2s}
\end{equation}
and
\begin{equation} \label{I8-2}
\|v\|_{L^{\gamma}(\Omega)} \leq C_5 \|h\|_{L^m(\Omega)}, \quad \forall\ 1 \leq \gamma < \infty, \quad \textup{ if } \quad \frac{N}{2s} \leq m < \frac{N}{2s-1}.
\end{equation}
Taking into account \eqref{I6-2}-\eqref{I8-2}, the result follows from \eqref{I5}.
\medbreak
\noindent 3) $m \geq \frac{N}{2s-t}$.
\medbreak
The result follows arguing as in the proof of Proposition \ref{regularity-bessel}, 3) using Proposition \ref{regu-grad1}, 2).
\end{proof}

\begin{proof}[\textbf{Proof of Corollary \ref{corollary regu-grad1}}]
By \cite[Lemma 15.9]{Ponce_Book_2015} we know that
\begin{equation}\label{fract-grad}
\n^s u(x)=\frac{1}{N-(1-s)}\irn \frac{\n u(y)}{|x-y|^{N+s-1)}}dy.
\end{equation}
Hence, we have that
\begin{equation} \label{ineq-fract-grad}
|\n^s u(x)|\le C \irn \frac{|\nabla u(y)|}{|x-y|^{N+s-1}}dy.
\end{equation}
The result then follows arguing as in the proof of Proposition \ref{regu-grad1}.
\end{proof}

\subsection{Convergence and compactness}\label{3.3} $ $
\medbreak
We end this section presenting a result of convergence and one of compactness for the fractional Poisson equation \eqref{linearEq}. They will be used in the proofs of Theorems \ref{main th existence}, \ref{main th existence map1}, \ref{main th Qlambda} and \ref{maint th Rlambda}.

\begin{prop} \label{convergence}
Let $\Omega \subset \RN$, $N \geq 2$, be a bounded domain with $\partial \Omega$ of class $\mathcal{C}^2$, let $s \in (1/2,1)$, let $\{h_n\} \subset L^1(\Omega)$ be a sequence such that $h_n \to h$ in $L^1(\Omega)$ and let $v_n$ be the unique weak solution to
\begin{equation*}
\left\{
\begin{aligned}
\sLap v_n & = h_n(x), \quad &  \textup{ in } \Omega,\\
v_n & = 0, & \textup{ in } \RN \setminus \Omega,
\end{aligned}
\right.
\end{equation*}
for all $n \in \N$, and $v$ be the unique weak solution to
\begin{equation*}
\left\{
\begin{aligned}
\sLap v & = h(x), \quad &  \textup{ in } \Omega,\\
v & = 0, & \textup{ in } \RN \setminus \Omega.
\end{aligned}
\right.
\end{equation*}
Then $v_n \to v$ in $W_0^{s,p}(\Omega)$ for all $1 \leq p < \frac{N}{N-s}$. 
\end{prop}

\begin{proof}
First of all observe that the existence of $v_n$ and $v$ are insured by Lemma \ref{LPPS weak solution regularity}. Now, let us define $w_n = v_n - v$ and observe that $w_n$ satisfies
\begin{equation*}
\left\{
\begin{aligned}
\sLap w_n & = h_n(x)-h(x), \quad &  \textup{ in } \Omega,\\
w_n & = 0, & \textup{ in } \RN \setminus \Omega.
\end{aligned}
\right.
\end{equation*}
Applying Proposition \ref{regularity} with $m = 1$, it follows that
\[ \|w_n\|_{W_0^{s,p}(\Omega)} \leq C_3 \|h_n - h\|_{L^1(\Omega)}, \quad \forall\ 1 \leq p < \frac{N}{N-s}.\]
Hence, since $h_n \to h$ in $L^1(\Omega)$, it follows that $w_n \to 0$ in $W_0^{s,p}(\Omega)$ for all $1 \leq p < \frac{N}{N-s}$ and so, that $v_n \to v$ in $W_0^{s,p}(\Omega)$ for all $1 \leq p < \frac{N}{N-s}$, as desired.
\end{proof}

\begin{prop} \label{compactness}
Let $\Omega \subset \RN$, $N \geq 2$, be a bounded domain with $\partial \Omega$ of class $\mathcal{C}^2$, let $s \in (1/2,1)$ and let $h \in L^1(\Omega)$. Then the operator $\mathcal{S}: L^1(\Omega) \to W_0^{s,p}(\Omega)$ given by $\mathcal{S}(h) = v$ with $v$ the unique weak solution to \eqref{linearEq} is compact for all $1 \leq p < \frac{N}{N-s}$.
\end{prop}

\begin{proof}
Let $\{f_n\} \subset L^1(\Omega)$  be a bounded sequence. By \cite[Proposition 2.4]{C_V_2014-JFA} we know that $\mathcal{S}$ is a compact operator from $L^1(\Omega)$ to $W_0^{1,p_1}(\Omega)$ for all $1 \leq \theta < \frac{N}{N-(2s-1)}$. Hence, for all $1 \leq \theta < \frac{N}{N-(2s-1)}$, up to a subsequence we have that $\mathcal{S}(f_n) \to v$ for some $v \in W_0^{1,\theta}(\Omega)$. By Sobolev inequality, this implies, for all $ 1 \leq \sigma < \frac{N}{N-2s}$, that $\mathcal{S}(f_n) \to v$ in $L^{\sigma}(\Omega)$ and $v \in L^{\sigma}(\Omega)$.
\medbreak
Now, applying Lemma \ref{BM interpolation} with $\eta = s$, $s_1 = 0$ and $s_2 =1$, we obtain that
\begin{equation} \label{compactness1}
\begin{aligned}
\|\mathcal{S}(f_n) - v\|_{W_0^{s,p}(\Omega)} & \leq C \|\mathcal{S}(f_n)-v\|_{L^{\sigma}(\RN)}^{1-s}\|\mathcal{S}(f_n)-v\|_{W^{1,\theta}(\RN)}^{s} \\
& = C \|\mathcal{S}(f_n)-v\|_{L^{\sigma}(\Omega)}^{1-s} \|\mathcal{S}(f_n)-v\|_{W_0^{1,\theta}(\Omega)}^s,
\end{aligned}
\end{equation}
for $p$ satisfying
\begin{equation}
\frac{1}{p} = \frac{1-s}{\sigma} + \frac{s}{\theta}.
\end{equation}
Hence, the result follows from \eqref{compactness1} using that
\[ 1 \geq \frac{1}{p} > \frac{N-s}{N}.\]
\end{proof}

\section{Proofs of Theorems \ref{main th existence} and \ref{main th existence map1}} \label{4}
This section is devoted to prove Theorems \ref{main th existence} and \ref{main th existence map1}. As indicated in the introduction, once we have the regularity results of Section \ref{3}, we follow the approach first develop in \cite[Section 6]{P_2014}. Let us begin with two elementary technical lemmas that will be useful in the proofs of both theorems.

\begin{lemma} \label{Lemma g}
Let $a, b > 0$, $p > 1$ and  $c^{\ast} := \frac{p-1}{p} \left( \frac{1}{pa^pb} \right)^{\frac{1}{p-1}}.$
Then, the function $g: [0,\infty) \to \R$ given by
\[ g(t) = a^p (bt + c^{\ast})^p - t,\]
has exactly one root $t^{\ast} \in (0,\infty)$. 
\end{lemma}

\begin{proof}
First observe that, $g'(t) = 0$ if and only if
\[ t = t^{\ast} := \frac{1}{b} \left( \frac{1}{pa^pb} \right)^{\frac{1}{p-1}} - \frac{c^{\ast}}{b} = \frac{1}{pb}  \left( \frac{1}{pa^pb} \right)^{\frac{1}{p-1}} \in (0,\infty)  .\]
Moreover, observe that
\[ g''(t^{\ast}) = (p-1)pa^pb^2 \left( \frac{1}{pa^pb} \right)^{\frac{p-2}{p-1}} > 0.\]
Thus, we deduce that $g$ has an strict global minimum on $t = t^{\ast}$. Finally, observe that
\[ g(t^{\ast})  =  a^p \left( \frac{1}{pa^pb} \right)^{\frac{p}{p-1}} - \frac{1}{b} \left( \frac{1}{pa^pb} \right)^{\frac{1}{p-1}} +  \frac{p-1}{pb} \left( \frac{1}{pa^pb} \right)^{\frac{1}{p-1}} = 0, \quad g(0) > 0 \textup{ and } \lim_{t \to \infty}g(t) = \infty.\] 
Hence, we conclude that $g$ has exactly one root $ t^{\ast} \in (0,\infty)$.
\end{proof}

\begin{lemma} \label{holder nonlocal gradient}
Let $\Omega \subset \RN$ be a bounded domain with boundary $\partial \Omega$ of class $\mathcal{C}^{0,1}$ and let $s \in (0,1)$. For all $\epsilon > 0$ satisfying $0 < s-\epsilon < s+\epsilon < 1$ and all $1 \leq \sigma < r$ there exists $C_9 = C_9(s,\epsilon,\sigma,r,\Omega) > 0$ such that
\begin{equation}
\left\| \left( \int_{\RN} \frac{|u(x)-u(y)|^{\sigma}}{|x-y|^{N+s\sigma}} dy \right)^{\frac{1}{\sigma}} \right\|_{L^r(\Omega)} \leq C_9 \|u\|_{W_0^{s+\epsilon,r}(\Omega)}\,, \quad \forall\ u \in W_0^{s+\epsilon,r}(\Omega)\,.
\end{equation}
\end{lemma}

\begin{proof}
First of all, observe that
\begin{equation} \label{J0}
\begin{aligned}
& \int_{\Omega}  \left( \int_{\RN} \frac{|u(x)-u(y)|^{\sigma}}{|x-y|^{N+s\sigma}} dy \right)^{\frac{r}{\sigma}} dx \\
&  \ \ = \int_{\Omega} \left( \int_{\RN \cap \{|x-y|< 1 \} } \frac{|u(x)-u(y)|^{\sigma}}{|x-y|^{N+s\sigma}} dy + \int_{\RN \cap \{ |x-y|\geq 1\} } \frac{|u(x)-u(y)|^\sigma}{|x-y|^{N+s\sigma}} dy \right)^{\frac{r}{\sigma}} dx \\
& \  \ \leq c_{r,\sigma} \left[ \int_{\Omega} \left( \int_{\RN \cap \{|x-y|< 1 \} } \frac{|u(x)-u(y)|^\sigma}{|x-y|^{N+s\sigma}} dy \right)^{\frac{r}{\sigma}} dx + \int_{\Omega} \left( \int_{\RN \cap \{|x-y| \geq 1 \} } \frac{|u(x)-u(y)|^\sigma}{|x-y|^{N+s\sigma}} dy \right)^{\frac{r}{\sigma}} dx \right] \\
& \ \ =: c_{r,\sigma} (J_1 + J_2)\,.
\end{aligned}
\end{equation}
Let us then estimate $J_1$. Applying H\"older inequality, we have that
\begin{equation*}
\begin{aligned}
J_1 & = \int_{\Omega}  \left( \int_{\RN \cap \{|x-y|< 1 \} } \frac{|u(x)-u(y)|^\sigma}{|x-y|^{\frac{N\sigma}{r} + (s+\epsilon) \sigma}} \frac{|x-y|^{\epsilon \sigma}}{|x-y|^{N - \frac{N\sigma}{r}}} dy \right)^{\frac{r}{\sigma}} dx\\
& \leq \int_{\Omega} \left( \int_{\RN \cap \{|x-y|< 1 \} } \frac{|u(x)-u(y)|^r}{|x-y|^{N+(s+\epsilon)r}} dy \right) \left( \int_{\RN \cap \{|x-y|< 1 \} }  \frac{dy}{|x-y|^{N-\frac{\epsilon \sigma r }{r-\sigma}}} \right)^{\frac{r-\sigma}{\sigma}} dx.
\end{aligned}
\end{equation*}
Furthermore, since
\[ \int_{\RN \cap \{|x-y|< 1 \} }  \frac{dy}{|x-y|^{N-\frac{\epsilon \sigma r }{r-\sigma}}} = \int_{B_1(0)} \frac{dz}{|z|^{N- \frac{\epsilon \sigma r }{r-\sigma}}} = C_{J_1}(\epsilon,\sigma,r) < \infty,\]
we deduce that
\begin{equation}\label{J1}
 J_1 \leq \widetilde{C}_{J_1} \int_{\Omega} \left( \int_{\RN \cap \{|x-y|< 1 \} } \frac{|u(x)-u(y)|^r}{|x-y|^{N+(s+\epsilon)r}} dy \right) dx \leq \widetilde{C}_{J_1} \|u\|_{W_0^{s+\epsilon,r}(\Omega)}^{r}  .
\end{equation}
Now, arguing as with $J_1$, we obtain that
\begin{equation*}
\begin{aligned}
J_2 & = \int_{\Omega}  \left( \int_{\RN \cap \{|x-y| \geq 1 \} } \frac{|u(x)-u(y)|^\sigma}{|x-y|^{\frac{N\sigma}{r} + (s-\epsilon) \sigma}} \frac{dy}{|x-y|^{N - \frac{N\sigma}{r} + \epsilon \sigma}} \right)^{\frac{r}{\sigma}} dx\\
& \leq \int_{\Omega} \left( \int_{\RN \cap \{|x-y|\geq 1 \} } \frac{|u(x)-u(y)|^r}{|x-y|^{N+(s-\epsilon)r}} dy \right) \left( \int_{\RN \cap \{|x-y| \geq 1 \} }  \frac{dy}{|x-y|^{N+\frac{\epsilon \sigma r }{r-\sigma}}} \right)^{\frac{r-\sigma}{\sigma}} dx.
\end{aligned}
\end{equation*}

\noindent Hence, since
\[ \int_{\RN \cap \{|x-y| \geq 1 \} }  \frac{dy}{|x-y|^{N+\frac{\epsilon \sigma r }{r-\sigma}}} = \int_{\RN \setminus B_1(0)} \frac{dz}{|z|^{N+ \frac{\epsilon \sigma r }{r-\sigma}}} = C_{J_2}(\epsilon,\sigma,r) < \infty,\]
and $W_{0}^{s+\epsilon,r}(\Omega) \subset W_{0}^{s-\epsilon,r}(\Omega)$, it follows that
\begin{equation} \label{J2}
J_2 \leq \widetilde{C}_{J_2} \int_{\Omega} \left( \int_{\RN \cap \{|x-y| \geq 1 \} } \frac{|u(x)-u(y)|^r}{|x-y|^{N+(s-\epsilon)r}} dy \right) dx \leq \widetilde{C}_{J_2} \|u\|_{W_0^{s-\epsilon,r}(\Omega)}^{r} \leq  \overline{C}_{J_2} \|u\|_{W_0^{s+\epsilon,r}(\Omega)}^{r}.
\end{equation}
The result follows from \eqref{J0}, \eqref{J1} and \eqref{J2}.
\end{proof}

\subsection{Proof of Theorem \ref{main th existence}} \label{4.1} $ $
\medbreak
Let us begin recalling that, under the assumption \eqref{A1}, $f \in L^m(\Omega)$ for some $m > \frac{N}{2s}$. Hence, since we are working in a bounded domain, without loss of generality, we can assume that $m \in (\frac{N}{2s}, \frac{N}{2s-1})$. Moreover, observe that, for $\lambda f \equiv 0$, $u \equiv 0$ is a solution to \eqref{Plambda} and, for $\mu \equiv 0$, \eqref{Plambda} reduces to \eqref{linearEq}. Hence, we may assume that $\|\mu\|_{\infty} \neq 0$ and $\|f\|_{L^m(\Omega)} \neq 0$.
\medbreak
Next, we fix some notation that we use throughout this subsection. First, we fix $r = r(m,s) > 0$ such that
\[ 1 < 2m < r < \frac{mN}{s(N-m(2s-1))},\]
and $\epsilon = \epsilon(r,m,s) > 0$ such that
\[ 1 < r < \frac{mN}{(s+\epsilon)(N-m(2s-1))} <  \frac{mN}{s(N-m(2s-1))}, \quad s+\epsilon < 1, \quad \textup{ and } s-\epsilon > \frac{1}{2}.\]
Also, we introduce and fix the constants $C_1$, given by Proposition \ref{corollary regularity}, 3) applied with $t = s+\epsilon$ and $p =r$, $C_{10}:= C_9^2 |\Omega|^{\frac{r-2m}{rm}}$, where $C_9$ is the constant given by Lemma \ref{holder nonlocal gradient}, and
\[ \lambda^{\ast} := \frac{1}{4\|f\|_{L^m(\Omega)} C_1^2 C_{10} \|\mu\|_{\infty}}.\]
By the definition of $\lambda^{\ast}$ and Lemma \ref{Lemma g}, we know there exists and unique $l \in (0,\infty)$ such that
\begin{equation} \label{eq l}
 C_1 (C_{10} \|\mu\|_{\infty} l + \lambda^{\ast} \|f\|_{L^m(\Omega)} ) = l^{\frac{1}{2}}.
\end{equation}
\bigbreak
Having fixed the above constants, we introduce
\[ E:= \left\{ v \in W_0^{s,1}(\Omega) : \iint_{D_{\Omega}} \frac{|u(x)-u(y)|^{r}}{|x-y|^{N+(s+\epsilon)r}} dx dy \leq l^{\frac{r}{2}} \right\}, \]
which is a closed convex set of $W_0^{s,1}(\Omega)$. Then, we define $T: E \to W_0^{s,1}(\Omega)$ by $T(\varphi) = u$, where $u$ is the weak solution to
\begin{equation} \label{eq fixed point}
\left\{
\begin{aligned}
(-\Delta)^s u & = \mu(x)\, \mathbb{D}_s^2(\varphi) + \lambda f(x)\,, & \quad \textup{ in } \Omega,\\
u & = 0\,, & \quad \textup{ in } \RN \setminus \Omega,
\end{aligned}
\right.
\end{equation}
and observe that problem \eqref{Plambda} is equivalent to the fixed point problem $u = T(u)$. Hence, to prove Theorem \ref{main th existence}, we shall show that $T$ has fixed point belonging to $W_0^{s,2}(\Omega) \cap \mathcal{C}^{0,\alpha}(\Omega)$ for some $\alpha > 0$.

\begin{lemma} \label{T well defined}
Assume that \eqref{A1} holds. Then $T$ is well defined.
\end{lemma}
\begin{proof}
First of all, by H\"older inequality and Lemma \ref{holder nonlocal gradient}, observe that for all $\varphi \in E$,
\begin{equation}
\begin{aligned} \label{wd0}
\int_{\Omega} \mathbb{D}_s^2(\varphi) dx \leq c(r,\Omega) \left( \int_{\Omega} (\mathbb{D}_s^2(\varphi))^{\frac{r}{2} } dx \right)^{\frac{2}{r}} \leq c\, C_9^2 \|\varphi\|_{W_0^{s+\epsilon,r}(\Omega)}^{2} = c\, C_9^2\, l.
\end{aligned}
\end{equation}
Hence, for all $\varphi \in E$, it follows that
\begin{equation} \label{wd}
\|\mu(x)\, \mathbb{D}_s^2(\varphi) + \lambda f(x)\|_{L^1(\Omega)} \leq c \, C_9^2 \|\mu\|_{\infty} l + |\lambda|\, \|f\|_{L^1(\Omega)} = C < \infty.
\end{equation}
Thanks to Lemma \ref{LPPS weak solution regularity} and Proposition \ref{regularity}, if the right hand side in \eqref{eq fixed point} belongs to $L^1(\Omega)$, problem \eqref{eq fixed point} has an unique weak solution and it belongs to $W_0^{s,1}(\Omega)$. Thus, the result follows from \eqref{wd}.
\end{proof}

\begin{lemma} \label{TE subset E}
Assume \eqref{A1} and let $0 < \lambda \leq \lambda^{\ast}$. Then $T(E) \subset E$.
\end{lemma}

\begin{proof}
For an arbitrary $\varphi  \in E$, we define $u = T(\varphi)$. Now, by Proposition \ref{corollary regularity} and since $0 < \lambda \leq \lambda^{\ast}$ , it follows that
\begin{equation} \label{ineq10}
\begin{aligned}
\left( \iint_{D_{\Omega}} \frac{|u(x)-u(y)|^r}{|x-y|^{N+(s+\epsilon)r}} dx dy \right)^{\frac{1}{r}} & \leq C_1 \big\| \mu(x) \mathbb{D}_s^2(\varphi) + \lambda f(x) \big \|_{L^m(\Omega)} \\
& \leq C_1 \|\mu\|_{\infty} \big \| \mathbb{D}_s^2(\varphi) \big \|_{L^m(\Omega)} + C_1 \lambda^{\ast} \|f\|_{L^m(\Omega)}.
\end{aligned}
\end{equation}
Also, by Lemma \ref{holder nonlocal gradient}, H\"older inequality and the definition of $C_{10}$, we obtain that
\begin{equation*}
\begin{aligned}
\|\mathbb{D}_s^2(\varphi) \|_{L^m(\Omega)}  \leq |\Omega|^{\frac{r-2m}{rm}} \|(\mathbb{D}_s^2(\varphi))^{\frac{1}{2}}\|_{L^{r}(\Omega)}^{2} \leq C_9^{2} |\Omega|^{\frac{r-2 m}{rm}} \|\varphi\|_{W_0^{s+\epsilon,r}(\Omega)}^{2} = C_{10}  \|\varphi\|_{W_0^{s+\epsilon,r}(\Omega)}^{2}.
\end{aligned}
\end{equation*} 
Thus, since $\varphi \in E$, we have that
\begin{equation} \label{ineq11}
\| \mathbb{D}_s^2(\varphi) \|_{L^m(\Omega)} \leq C_{10}\, l.
\end{equation}
From \eqref{eq l}, \eqref{ineq10} and \eqref{ineq11}, it follows that
\[ \left( \iint_{D_{\Omega}} \frac{|u(x)-u(y)|^r}{|x-y|^{N+(s+\epsilon)r}} dx dy \right)^{\frac{1}{r}} \leq C_1 (C_{10} \|\mu\|_{\infty}\, l + \lambda^{\ast} \|f\|_{L^m(\Omega)} )  = l^{\frac{1}{2}}. \]
Hence, since by Proposition \ref{regularity} we also know that $u \in W_0^{s,1}(\Omega)$, we conclude that $u \in E$ and so, that $T(E) \subset E$.
\end{proof}

\begin{lemma} \label{T continuous}
Assume that \eqref{A1} holds. Then $T$ is continuous.
\end{lemma}

\begin{proof}
Let $\{\varphi_n\} \subset E$ be a sequence such that $\varphi_n \to \varphi$ in $W_0^{s,1}(\Omega)$ and define $u_n = T(\varphi_n)$, for all $n \in \N$, and $u = T(\varphi)$. To show that $u_n \to u$ in $W_0^{s,1}(\Omega)$, and so, that $T$ is continuous, we prove that
\begin{equation} \label{L1convergence}
 g_n(x) := \mathbb{D}_s^2(\varphi_n) + \lambda f(x) \to g(x) := \mathbb{D}_s^2(\varphi) + \lambda f(x), \quad \textup{ in } L^1(\Omega).
\end{equation}
Indeed, if \eqref{L1convergence} holds, the result follows from Proposition \ref{convergence}.
\medbreak
First of all, using the notation $\psi_n = \varphi_n -\varphi$ and the reverse triangle inequality, we obtain that
\begin{equation*} \label{ineq13}
\begin{aligned}
\| \mathbb{D}_s^2(  \varphi_n) - \mathbb{D}_s^2(\varphi) \|_{L^1(\Omega)}  & = \int_{\Omega} \left|  \int_{\RN} \frac{|\varphi_n(x)-\varphi_n(y)|^2 - |\varphi(x) -\varphi(y)|^2}{|x-y|^{N +2s}} dy  \right| dx \\
& \leq \int_{\Omega} \left| \int_{\RN} \frac{ \big(|\varphi_n(x)-\varphi_n(y)| + |\varphi(x)-\varphi(y)| \big)|\psi_n(x)-\psi_n(y)|}{|x-y|^{N+2s}} dy \right| dx \\
& = \int_{\Omega} \left( \int_{\RN} \frac{ |\varphi_n(x)-\varphi_n(y)| + |\varphi(x)-\varphi(y)| }{|x-y|^{\frac{N}{2}+s}} \cdot \frac{|\psi_n(x)-\psi_n(y)|}{|x-y|^{\frac{N}{2}+ s}} dy \right) dx.
\end{aligned}
\end{equation*}
Applying then H\"older inequality, we deduce that
\begin{equation*} \label{ineq14}
\begin{aligned}
\| \mathbb{D}_s^2(  \varphi_n) & - \mathbb{D}_s^2(\varphi) \|_{L^1(\Omega)}  \\
& \leq  \int_{\Omega} \left( \int_{\RN} \frac{(|\varphi_n(x)-\varphi_n(y)| + |\varphi(x)-\varphi(y)|)^2 }{|x-y|^{N+2s}} dy \right)^{\frac{1}{2}} \left( \int_{\RN} \frac{|\psi_n(x)-\psi_n(y)|^2}{|x-y|^{N+2s}} dy \right)^{\frac{1}{2}}dx \\
& \leq \left( \int_{\Omega} \left( \int_{\RN} \frac{(|\varphi_n(x)-\varphi_n(y)| + |\varphi(x)-\varphi(y)|)^2 }{|x-y|^{N+2s}} dy \right) dx \right)^{\frac{1}{2}} \left( \int_{\Omega} \left( \int_{\RN}  \frac{|\psi_n(x)-\psi_n(y)|^2}{|x-y|^{N+2s}} dy \right) dx \right)^{\frac{1}{2}} \\
& = \left( \int_{\Omega} \left( \int_{\RN} \frac{(|\varphi_n(x)-\varphi_n(y)| + |\varphi(x)-\varphi(y)|)^2}{|x-y|^{N+2s}} dy \right) dx \right)^{\frac{1}{2}} \|\mathbb{D}_s^2(\varphi_n - \varphi)\|_{L^1(\Omega)}^{\frac{1}{2}} \\
& =:  I_1 \cdot I_2\,.
\end{aligned}
\end{equation*}
Taking into account the above inequality, if we show that $I_1$ is bounded and $I_2$ goes to zero, we deduce that $\| \mathbb{D}_s^2(\varphi_n) - \mathbb{D}_s^2(\varphi)\|_{L^1(\Omega)} \to 0$.
\medbreak
\noindent \textbf{Claim 1}: \textit{$I_1$ is bounded.}
\medbreak
Directly observe that
\begin{equation}
\begin{aligned}
I_1 & \leq 2 \left[ \int_{\Omega} \left( \int_{\RN} \frac{|\varphi_n(x)-\varphi_n(y)|^2}{|x-y|^{N+2s}} dy \right) dx + \int_{\Omega}  \left( \int_{\RN} \frac{|\varphi(x)-\varphi(y)|^2}{|x-y|^{N+2s}} dy \right) dx \right] \\
&  = 2  \Big[ \|\mathbb{D}_s^2(\varphi_n)\|_{L^1(\Omega)} + \|\mathbb{D}_s^2(\varphi)\|_{L^1(\Omega)} \Big].
\end{aligned}
\end{equation}
Since $\varphi_n,\ \varphi \in E$ for all $n \in \N$, by \eqref{wd0}, we have that
\[  \Big[ \|\mathbb{D}_s^2(\varphi_n)\|_{L^1(\Omega)} + \|\mathbb{D}_s^2(\varphi)\|_{L^1(\Omega)} \Big] \leq 2 c\, C_9^2 \, l < \infty,\]
and so, that $I_1$ is bounded.
\medbreak
\noindent \textbf{Claim 2}:\textit{ $I_2$ goes to zero.}
\medbreak Let $\theta \in (0,1)$ be small enough to ensure that
$\frac{2-\theta}{1-\theta} < r$. By H\"older inequality, it follows that
\begin{equation}
\begin{aligned}
\| \mathbb{D}_s^2(\varphi_n & - \varphi)\|_{L^1(\Omega)} = \int_{\Omega} \left( \int_{\RN} \frac{|\psi_n(x)-\psi_n(y)|^2}{|x-y|^{N+2s}} dy \right) dx \\
& = \int_{\Omega} \left( \int_{\RN} \frac{|\psi_n(x)-\psi_n(y)|^{\theta}}{|x-y|^{(N+s)\theta}} \frac{|\psi_n(x)-\psi_n(y)|^{2-\theta}}{|x-y|^{N(1-\theta)+s(2-\theta)}} dy \right) dx \\
& \leq \int_{\Omega} \left[ \left( \int_{\RN} \frac{|\psi_n(x)-\psi_n(y)|}{|x-y|^{N+s}} dy \right)^{\theta} \left( \int_{\RN} \frac{|\psi_n(x)-\psi_n(y)|^{\frac{2-\theta}{1-\theta}}}{|x-y|^{N+s \frac{2-\theta}{1-\theta}}} dy \right)^{1-\theta} \right] dx \\
& \leq \left( \int_{\Omega} \left( \int_{\RN} \frac{|\psi_n(x)-\psi_n(y)|}{|x-y|^{N+s}} dy \right) dx \right)^{\theta} \left( \int_{\Omega} \left( \int_{\RN} \frac{|\psi_n(x)-\psi_n(y)|^{\frac{2-\theta}{1-\theta}}}{|x-y|^{N+s \frac{2-\theta}{1-\theta}}} dy \right) dx \right)^{1-\theta}.
\end{aligned}
\end{equation}
Hence, since $\varphi_n \to \varphi$ in $W_0^{s,1}(\Omega)$ implies that
\begin{equation} \label{convergeW01}
  \int_{\Omega} \left( \int_{\RN} \frac{|\psi_n(x)-\psi_n(y)|}{|x-y|^{N+s}} dy \right) dx \to 0,
\end{equation}
if we prove that
\begin{equation} \label{bounded}
\int_{\Omega} \left( \int_{\RN} \frac{|\psi_n(x)-\psi_n(y)|^{\frac{2-\theta}{1-\theta}}}{|x-y|^{N+s \frac{2-\theta}{1-\theta}}} dy \right) dx
\end{equation}
is bounded, we can conclude that $I_2$ goes to zero, as desired. Since we have chosen $\theta \in (0,1)$ small enough in order to ensure that $\frac{2-\theta}{1-\theta} < r$ and $\Omega$ is a bounded domain, it follows that
\begin{equation} \label{ineq15}
\begin{aligned}
\int_{\Omega} \left( \int_{\RN} \frac{|\psi_n(x)-\psi_n(y)|^{\frac{2-\theta}{1-\theta}}}{|x-y|^{N+s \frac{2-\theta}{1-\theta}}} dy \right) dx \leq C(r,\Omega) \left( \int_{\Omega} \left( \int_{\RN} \frac{|\psi_n(x)-\psi_n(y)|^{\frac{2-\theta}{1-\theta}}}{|x-y|^{N+s\frac{2-\theta}{1-\theta}}} dy \right)^{\frac{r}{\frac{2-\theta}{1-\theta}}} dx \right)^{\frac{\frac{2-\theta}{1-\theta}}{r}}.
\end{aligned}
\end{equation}
Applying then Lemma \ref{holder nonlocal gradient} and the triangular inequality we have that
\begin{equation}
\begin{aligned}
\int_{\Omega} \left( \int_{\RN} \frac{|\psi_n(x)-\psi_n(y)|^{\frac{2-\theta}{1-\theta}}}{|x-y|^{N+s \frac{2-\theta}{1-\theta}}} dy \right) dx & \leq \overline{C} \|\psi_n\|_{W_0^{s+\epsilon,r}(\Omega)} \\
 & \leq \widetilde{C} \big[ \|\varphi_n\|_{W_0^{s+\epsilon,r}(\Omega)} + \|\varphi\|_{W_0^{s+\epsilon,r}(\Omega)} \big] \\
 & \leq  2\, \widetilde{C}\, l^{\frac{1}{2}} = \widehat{C},
\end{aligned}
\end{equation}
where $\overline{C}, \widetilde{C}$ and $\widehat{C}$ are positive constants independent of $n$. Thus, we conclude that \eqref{bounded} is indeed bounded.
\medbreak
From Claim 1 and 2 we deduce that $\| \mathbb{D}_s^2(\varphi_n) - \mathbb{D}_s^2(\varphi) \|_{L^1(\Omega)} \to 0$. This implies that $g_n \to g$ in $L^1(\Omega)$, as desired, and the result follows.
\end{proof}

\begin{lemma} \label{T compact}
Assume that \eqref{A1} holds. Then $T$ is compact.
\end{lemma}

\begin{proof}
Let $\{\varphi_n\} \subset E$ be a bounded sequence in $W_0^{s,1}(\Omega)$ and define $u_n = T(\varphi_n)$ for all $n \in \N$. We have to show that $u_n \to u$ in $W_0^{s,1}(\Omega)$ for some $u \in W_0^{s,1}(\Omega)$.
\medbreak
Since $\{\varphi_n\} \subset E$ for all $n \in \N$, arguing as in Lemma \ref{T well defined}, we deduce that $\{\mathbb{D}_s^2(\varphi_n)\}$ is a bounded sequence in $L^1(\Omega)$. Hence, if we define
\[ g_n(x):= \mathbb{D}_s^2(\varphi_n) + \lambda f(x), \quad \forall\ n \in \N,\]
we have that $\{g_n\}$ is a bounded sequence in $L^1(\Omega)$. The result then follows from Proposition \ref{compactness}.
\end{proof}

\begin{proof}[\textbf{Proof of Theorem \ref{main th existence}}]
Since $E$ is a closed convex set of $W_0^{s,1}(\Omega)$ and, by Lemmas  \ref{T well defined}, \ref{TE subset E}, \ref{T continuous}, and \ref{T compact}, we know that $T$ is continuous, compact and satisfies $T(E) \subset E$, we can apply the Schauder fixed point Theorem to obtain $u \in E$ such that $T(u) = u$. Thus, we conclude that \eqref{Plambda} has a weak solution for all $0 < \lambda \leq \lambda^{\ast}$. Finally, since $u \in W_0^{s,1}(\Omega) \cap W_0^{s,r}(\Omega)$ for some $1 < 2 < r$, by Lemma \ref{BM interpolation} applied with $s_1 = s_2 = s$, we deduce that $u \in W_0^{s,2}(\Omega)$. Moreover, since $r > N/s$,  by \cite[Theorem 8.2]{DN_P_V_2012}, we know that every $\varphi \in E$ belongs to $\mathcal{C}^{0,\alpha}(\Omega)$ for some $\alpha > 0$. 
\end{proof}

\vspace{0.05cm}

\subsection{Proof of Theorem \ref{main th existence map1}} $ $
\medbreak
First observe that, as before, without loss of generality we can assume $m \in (\frac{N}{2s}, \frac{N}{s})$, $\|\mu\|_{\infty} \neq 0$ and $\|f\|_{L^m(\Omega)} \neq 0$. Next, let us fix some notation. We fix
\begin{equation} \label{rmap1}
r = \frac{3mN}{N+ms}
\end{equation}
and $\epsilon := \epsilon(r,m,s) > 0$ such that
\[ 1 < r < \frac{mN}{N-m(s-\epsilon)} < \frac{mN}{N-ms}, \quad s+\epsilon < 1 \quad \textup{ and } s-\epsilon > \frac{1}{2}.\]
Also, we introduce and fix the constants $C_2$, given by Corollary \ref{corollary regularity-3} applied with $t = s+\epsilon$ and $p = r$, $C_{11}:= S_{N,r} C C_9^{2m}$, where $S_{N,r}$ is  the optimal constant in the Sobolev inequality (Theorem \ref{Sobolev inequality}), $C$ is the smallest constant guaranteeing the continuous embedding $W_0^{s+\epsilon,r}(\Omega) \subset W_0^{s,r}(\Omega)$ and $C_9$ is the constant given by Lemma \ref{holder nonlocal gradient}, and
\[ \lambda^{\ast} := \frac{2}{3\|f\|_{L^m(\Omega),}} \left( \frac{1}{3 C_2^3 C_{11} \|\mu\|_{\infty} } \right)^{\frac{1}{2}}. \]
Then, by Lemma \ref{Lemma g} we know that there exists and unique $l \in (0,\infty)$ such that
\begin{equation} \label{lmap1}
C_2 (C_{11} \|\mu\|_{\infty} l + \lambda^{\ast} \|f\|_{L^m(\Omega)} ) = l^{\frac{1}{3}}.
\end{equation}
\medbreak
Having fixed all these constants, we define
\[ E_1:= \left\{ v \in W_0^{s,1}(\Omega) : \iint_{D_{\Omega}} \frac{|u(x)-u(y)|^{r}}{|x-y|^{N+(s+\epsilon)r}} dx dy \leq l^{\frac{r}{3}} \right\}, \]
which is a closed convex set of $W_0^{s,1}(\Omega)$, and $T_1: E_1 \to W_0^{s,1}(\Omega)$ by $T_1(\varphi) = u$, with $u$ the unique weak solution to
\begin{equation} \label{map2}
\left\{
\begin{aligned}
(-\Delta)^s u & = \mu(x)\, \varphi \, \mathbb{D}_s^2(\varphi) + \lambda f(x)\,, & \quad \textup{ in } \Omega,\\
u & = 0\,, & \quad \textup{ in } \RN \setminus \Omega.
\end{aligned}
\right.
\end{equation}
Observe that \eqref{map1} is equivalent to the fixed point problem $u = T_1(u)$. Hence, we shall prove that $T_1$ has a fixed point belonging to $W_0^{s,2}(\Omega) \cap \mathcal{C}^{0,\alpha}(\Omega)$ for some $\alpha > 0$.

\begin{lemma} \label{lemma1map1}
For all $\varphi \in W_0^{s+\epsilon,r}(\Omega)$, it follows that
\begin{equation} \label{estimatemap1}
\|\varphi\, \mathbb{D}_s^2(\varphi)\|_{L^m(\Omega)} \leq C_{11} \|\varphi\|_{W_0^{s+\epsilon,r}(\Omega)}^3.
\end{equation}
\end{lemma}

\begin{proof}
First observe that, with the above notation, we have that
\[ 2 < \frac{2mr_s^{\ast}}{r_s^{\ast}-m} = r.\]
Hence, by H\"older and Sobolev inequalities and using that $W^{s+\epsilon,r}_0(\Omega) \subset W_0^{s,r}(\Omega)$ with continuous inclusion, we obtain that
\[ \|\varphi\, \mathbb{D}_s^2(\varphi)\|_{L^m(\Omega)}^m \leq S_{N,r} C\, \|\varphi\|_{W_0^{s+\epsilon,r}(\Omega)}^m \left\| \left( \int_{\RN} \frac{|\varphi(x)-\varphi(y)|^2}{|x-y|^{N+2s}} dy \right)^{\frac{1}{2}} \right\|_{L^{r}(\Omega)}^{2m}.\]
Since $r > 2$, the result follows from Lemma \ref{holder nonlocal gradient}.
\end{proof}

\begin{cor} \label{T1 well defined}
Assume that \eqref{A1} holds. Then $T_1$ is well defined.
\end{cor}

\begin{proof}
Since $\Omega$ is a bounded domain and $m > \frac{N}{2s} > 1$ the result follows from Lemma \ref{lemma1map1} arguing as in the proof of Lemma \ref{T well defined}.
\end{proof}

\begin{lemma} \label{T1 subset T1}
Assume \eqref{A1} and let $0 < \lambda \leq \lambda^{\ast}$. Then $T_1(E_1) \subset E_1$.
\end{lemma}

\begin{proof}
Let us consider an arbitrary $\varphi \in E$ and define $u = T_1(\varphi)$. By Corollary \ref{corollary regularity-3}, since that $0 < \lambda \leq \lambda^{\ast}$, we have that
\begin{equation} \label{ineq10map1}
\begin{aligned}
\left( \iint_{D_{\Omega}} \frac{|u(x)-u(y)|^r}{|x-y|^{N+(s+\epsilon)r}} dx dy \right)^{\frac{1}{r}}  \leq C_2 \|\mu\|_{\infty} \big \|\varphi\, \mathbb{D}_s^2(\varphi) \big \|_{L^m(\Omega)} + C_2 \lambda^{\ast} \|f\|_{L^m(\Omega)}.
\end{aligned}
\end{equation}
Hence, since $\varphi \in E$, by Lemma \ref{lemma1map1} and \eqref{lmap1}, it follows that
\[ \left( \iint_{D_{\Omega}} \frac{|u(x)-u(y)|^r}{|x-y|^{N+(s+\epsilon)r}} dx dy \right)^{\frac{1}{r}} \leq C_2 (C_{11} \|\mu\|_{\infty}\, l + \lambda^{\ast} \|f\|_{L^m(\Omega)} )  = l^{\frac{1}{3}}. \]
Thus, as by Proposition \ref{regularity} we also know that $u \in W_0^{s,1}(\Omega)$, we conclude that $u \in E_1$ and so, that $T_1(E_1) \subset E_1$.
\end{proof}

\begin{lemma} \label{T1 continuous}
Assume \eqref{A1}. Then $T_1$ is continuous.
\end{lemma}

\begin{proof}
Let $\{\varphi_n\} \subset E$ be a sequence such that $\varphi_n \to \varphi$ in $W_0^{s,1}(\Omega)$ and define $u_n = T_1(\varphi_n)$, for all $n \in \N$, and $u = T_1(\varphi)$. Arguing as in the proof of Lemma \ref{T continuous}, we just have to prove that
\begin{equation} \label{L1 convergence harmonic maps}
\varphi_n \mathbb{D}_s^2(\varphi_n) \to \varphi\, \mathbb{D}_s^2(\varphi), \quad \textup{ in } L^1(\Omega).
\end{equation}
\medbreak
First observe that, since $r > \frac{N}{s} > \frac{N}{s+\epsilon}$, for all $\varphi \in E_1$, it follows that
\begin{equation}
\|\varphi\|_{\Linfty} \leq C \iint_{D_{\Omega}} \frac{|\varphi(x)-\varphi(y)|^r}{|x-y|^{N+(s+\epsilon)r}} dx dy \leq l^{\frac{r}{3}}.
\end{equation}
Hence, since $\varphi_n \to \varphi$ in $W_0^{s,1}(\Omega)$, by Vitali's Convergence Theorem we deduce that $\varphi_n \to \varphi$ in $L^{\alpha}(\Omega)$ for all $1 \leq \alpha < \infty$.
\medbreak
Next, observe that
\begin{equation}
\begin{aligned}
\| \varphi_n \mathbb{D}_s^2(\varphi_n) - \varphi \mathbb{D}_s^2(\varphi)\|_{L^1(\Omega)} & = \| \varphi_n (\mathbb{D}_s^2(\varphi_n) - \mathbb{D}_s^2(\varphi)) + \mathbb{D}_s^2(\varphi) (\varphi_n - \varphi) \|_{L^1(\Omega)} \\
& \leq \| \varphi_n (\mathbb{D}_s^2(\varphi_n) - \mathbb{D}_s^2(\varphi))\|_{L^1(\Omega)} + \| \mathbb{D}_s^2(\varphi) (\varphi_n - \varphi) \|_{L^1(\Omega)}\\
& \leq  \|\varphi_n\|_{\infty} \|\mathbb{D}_s^2(\varphi_n) - \mathbb{D}_s^2(\varphi)\|_{L^1(\Omega)} + \| \mathbb{D}_s^2(\varphi)\|_{L^m(\Omega)} \|\varphi_n - \varphi\|_{L^{m'}(\Omega)} \\
& =: I_1 + I_2
\end{aligned}
\end{equation}
Then, arguing exactly as in Lemma \ref{T continuous} and using that $\|\varphi_n\|_{\infty} \leq C$ (independent of $n$) we deduce that $I_1 \to 0$. On the other hand, we know that $\|\mathbb{D}_s^2(\varphi)\|_{L^m(\Omega)} < \infty$. Hence, since $\varphi_n \to \varphi$ in $L^{\alpha}(\Omega)$ for all $1 \leq \alpha < \infty$, we also obtain that $I_2 \to 0$. We then conclude that \eqref{L1 convergence harmonic maps} holds, as desired.
\end{proof}

\begin{proof}[\textbf{Proof of Theorem \ref{main th existence map1}}]
Observe that the compactness of $T_1$ follows arguing exactly as in Lemma \ref{T compact}. Hence, since $E_1$ is a closed convex set of $W_0^{s,1}(\Omega)$ and, by Lemmas \ref{T1 subset T1}, \ref{T1 well defined} and \ref{T1 continuous} we know that $T_1$ is well defined, continuous and satisfies $T_1(E_1) \subset E_1$, we can apply the Schauder fixed point Theorem to obtain $u \in E_1$ such that $T_1(u) = u$. Thus, we conclude that \eqref{map1} has a weak solution for all $0 < \lambda \leq \lambda^{\ast}$. Finally, since $u \in W_0^{s,1}(\Omega) \cap W_0^{s,r}(\Omega)$ for some $1 < 2 < r$, by Lemma \ref{BM interpolation} we deduce that $u \in W_0^{s,2}(\Omega)$. Moreover, since $r > N/s$,  by \cite[Theorem 8.2]{DN_P_V_2012}, we know that every $\varphi \in E_1$ belongs to $\mathcal{C}^{0,\alpha}(\Omega)$ for some $\alpha > 0$. 
\end{proof}

\section{Proofs of Theorems \ref{main th non-existence} and \ref{optimality}} \label{5}

In this section we prove Theorems \ref{main th non-existence} and \ref{optimality}. The aim of these theorems is to justify the hypotheses considered in Theorem \ref{main th existence}. First we prove that \eqref{Plambda} has no solutions for $\lambda$ large and so, that the smallness condition is somehow necessary to have existence of solution.

\begin{proof}[\textbf{Proof of Theorem \ref{main th non-existence}}]
Assume that \eqref{Plambda} has a solution $u \in W_0^{s,2}(\Omega)$ and let $\phi \in \mathcal{C}_0^{\infty}(\Omega)$ be an arbitrary function such that
\[ \int_{\Omega} f(x) \phi^2(x) dx > 0,\]
Considering $\phi^{2}$ as test function in \eqref{Plambda} we observe that
\begin{equation} \label{eq 4.1}
\int_{\Omega} \sLap u \, \phi^2(x) dx = \int_{\Omega} \mu(x) \int_{\RN} \frac{|u(x)-u(y)|^{2}}{|x-y|^{N+2s}} \phi^{2}(x) dy dx + \lambda \int_{\Omega} f(x) \phi^2(x)dx.
\end{equation}
Now, on one hand, since $\mu(x) \geq \mu_1 > 0$ and $\mathbb{D}_s^2$ is symmetric in $x,y$, it follows that
\begin{equation} \label{eq 4.2}
\begin{aligned}
\int_{\Omega} \mu(x) & \int_{\RN} \frac{|u(x)-u(y)|^{2}}{|x-y|^{N+2s}} \phi^2(x) dy dx = \iint_{D_{\Omega}} \mu(x) \frac{|u(x)-u(y)|^{2}}{|x-y|^{N+2s}} \phi^2(x) dy dx  \\
& \geq \mu_1 \iint_{D_{\Omega}}  \frac{|u(x)-u(y)|^{2}}{|x-y|^{N+2s}} \phi^2(x) dy dx  \\
& = \frac{\mu_1}{2} \iint_{D_{\Omega}} \frac{|u(x)-u(y)|^{2}}{|x-y|^{N+2s}} \phi^2(x) dy dx   + \frac{\mu_1}{2} \iint_{D_{\Omega}} \frac{|u(x)-u(y)|^{2}}{|x-y|^{N+2s}} \phi^2(y) dy dx  \\
& \geq  \frac{\mu_1}{4} \iint_{D_{\Omega}} \frac{|u(x)-u(y)|^{2}}{|x-y|^{N+2s}} (\phi(x)+\phi(y))^{2} dy dx.
\end{aligned}
\end{equation}
On the other hand, by Young's inequality, it follows that
\begin{equation} \label{eq 4.3}
\begin{aligned}
\int_{\Omega} \sLap u\, \phi^2(x)dx & = \iint_{D_{\Omega}} \frac{(u(x)-u(y))(\phi^{2}(x) - \phi^{2}(y))}{|x-y|^{N+2s}} dy dx\\
& = \iint_{D_{\Omega}} \frac{(u(x)-u(y))(\phi(x)-\phi(y))(\phi(x)+\phi(y))}{|x-y|^{N+2s}} dydx \\
& \leq \iint_{D_{\Omega}} \frac{|u(x)-u(y)| |\phi(x)+\phi(y)|}{|x-y|^{ \frac{N}{2}+s}} \cdot \frac{|\phi(x)-\phi(y)|}{|x-y|^{ \frac{N}{2}+s}} dydx \\
& \leq \frac{\mu_1}{4} \iint_{D_{\Omega}} \frac{|u(x)-u(y)|^{2}}{|x-y|^{N+2s}} (\phi(x)+\phi(y))^{2} dy dx + \frac{1}{\mu_1} \iint_{D_{\Omega}} \frac{|\phi(x)-\phi(y)|^{2}}{|x-y|^{N+2s'}} dy dx
\end{aligned}
\end{equation}
Hence, substituting \eqref{eq 4.2} and \eqref{eq 4.3} to \eqref{eq 4.1}, we deduce that, if \eqref{Plambda} has a solution, then
\begin{equation} \label{last1}
\frac{1}{\mu_1} \iint_{D_{\Omega}} \frac{|\phi(x)-\phi(y)|^{2}}{|x-y|^{N+2s}} dy dx \geq \lambda \int_{\Omega} f(x) \phi^2(x) dx,
\end{equation}
which gives a contradiction for $\lambda$ large enough.
\end{proof}

\medbreak

Now, we prove Theorem \ref{optimality}. This theorem shows that the regularity considered on $f$ is almost optimal. Just the limit case $f \in L^{\frac{N}{2s}}(\Omega)$ remains open.

\begin{proof}[\textbf{Proof of Theorem \ref{optimality}}]
Without loss of generality we choose a bounded domain $\Omega$ with boundary $\partial \Omega$ of class $\mathcal{C}^2$ such that $0 \in \Omega$. Consider then
\begin{equation} \label{op1}
f(x) = \frac{1}{|x|^{\frac{N-\epsilon}{m}}},
\end{equation}
for some $\epsilon \in (0,1)$ to be chosen later and observe that, since $\Omega$ is bounded, $f \in L^m(\Omega)$.
\medbreak
We assume by contradiction that, for all $\epsilon > 0$, there exists $\lambda_{\epsilon} > 0$ such that \eqref{Plambda} has a solution $u \in W_0^{s,2}(\Omega)$. Arguing as in the proof of Theorem \ref{main th non-existence}, we conclude that, for all $\phi \in \mathcal{C}_0^{\infty}(\Omega) \setminus \{0\}$,
\begin{equation} \label{op2}
\begin{aligned}
\frac{1}{\mu_1} \iint_{D_{\Omega}}\frac{|\phi(x)-\phi(y)|^{2}}{|x-y|^{N+2s}} dy dx \geq \lambda_\e \int_{\Omega} f(x) \phi^2(x)dx =  \lambda_{\epsilon} \int_{\Omega} \frac{\phi^2(x)}{|x|^{\frac{N-\epsilon}{m}}} dx.
\end{aligned}
\end{equation}
Thus, we deduce that
\begin{equation} \label{op3}
0 < \mu_1\lambda_{\epsilon} \inf \left\{ \dfrac{\dyle \iint_{D_{\Omega}}\frac{|\phi(x)-\phi(y)|^{2}}{|x-y|^{N+2s}} dy dx}{\dyle\int_{\Omega} \frac{\phi^2(x)}{|x|^{\frac{N-\e}{m}}}dx}: \phi \in \mathcal{C}_0^{\infty}(\Omega) \setminus \{0\} \right\}.
\end{equation}
Nevertheless, since $m < \frac{N}{2s}$, we can choose $\epsilon > 0$ small enough to ensure that $\frac{N-\epsilon}{m} > 2s$. In that case, by Proposition \ref{cor_Hardy}, 2), we have that
\[ \inf \left\{ \dfrac{\dyle \iint_{D_{\Omega}}\frac{|\phi(x)-\phi(y)|^{2}}{|x-y|^{N+2s}} dy dx}{\dyle\int_{\Omega} \frac{|\phi(x)|^{2}}{|x|^{\frac{N-\e}{m}}}dx}: \phi \in \mathcal{C}_0^{\infty}(\Omega) \setminus \{0\} \right\} = 0,\]
which contradicts \eqref{op3}. Hence, the result follows.
\end{proof}

\section{Proofs of Theorems \ref{main th Qlambda} and \ref{maint th Rlambda}} \label{6}

This section is devoted to the proofs of Theorems \ref{main th Qlambda} and \ref{maint th Rlambda}. First, having at hand Proposition \ref{regu-grad1}, we prove Theorem \ref{main th Qlambda}  using again a fixed point argument. The proof is similar to the ones performed in Section \ref{4}. Hence, we skip some details. 
\medbreak
Since $\Omega$ is bounded, without loss of generality, we assume that $1 \leq m < \frac{N}{s}$. Also, if $\lambda f \equiv 0$, it follows that $u \equiv 0$ is a solution to \eqref{Qlambda} and, if $\mu \equiv 0$, \eqref{Qlambda} reduces to \eqref{linearEq}. Hence, we may also assume that $\|\mu\|_{\infty} \neq 0$ and $\|f\|_{L^m(\Omega)} \neq 0$.
\medbreak
Next, we fix some notation that will be used throughout the section. First, we fix $r = r(m,s,q) > 0$ such that
\[ 1 < qm < r < \frac{mN}{N-ms},\]
$C_3$ the constant given by Proposition \ref{regu-grad1} with $p = r$ and
\[ \lambda^{\ast} = \frac{q-1}{q\|f\|_{L^m(\Omega)}} \left( \frac{1}{q C_3^q |\Omega|^{\frac{r-qm}{r}} \|\mu\|_{\infty} } \right)^{\frac{1}{q-1}}.\]
Then, by the definition of $\lambda^{\ast}$ and Lemma \ref{Lemma g}, we know that there exists an unique $l \in (0,\infty)$ such that
\begin{equation} \label{l section 5}
C_3 ( \|\mu\|_{L^{\infty}(\Omega)} |\O|^{\frac{r-qm}{mr}} l + \lambda^{\ast} \|f\|_{L^m(\Omega)} ) = l^{\frac{1}{q}}.
\end{equation}
With the above constants fixed, we introduce
\[ E_2:= \left\{ v \in W_0^{s,1}(\Omega): \|(-\Delta)^{\frac{s}{2}}v\|_{L^r(\Omega)} \leq l^{\frac{1}{q}} \right\},\]
and observe that $E_2$ is a closed convex set of $W_0^{s,1}(\O)$. Then, we define $T_2: E_2 \to W_0^{s,1}(\Omega)$ by $T_2(\varphi) = u$ with $u$ the unique weak solution to
\begin{equation} \label{fixed point problem section 5}
\left\{
\begin{aligned}
(-\Delta)^s u & = \mu(x)|(-\Delta)^{\frac{s}{2}} \varphi|^q + \lambda f(x)\,, & \quad \textup{ in } \Omega,\\
u & = 0\,, & \quad \textup{ in } \RN \setminus \Omega,
\end{aligned}
\right.
\end{equation}
and observe that \eqref{Qlambda} is equivalent to the fixed point problem $u = T_2(u)$. Hence, we shall show that $T_2$ has a fixed point.
\medbreak
\begin{lemma} \label{T2 wd subset compact}
Assume \eqref{B1} and let $0 < \lambda \leq \lambda^{\ast}$. Then $T_2$ is well defined, $T_2(E_2) \subset E_2$ and  $T_2$ is compact.
\end{lemma}

\begin{proof}
The proof of this lemma follows as in Lemmas \ref{T well defined}, \ref{TE subset E} and \ref{T compact} using Proposition \ref{regu-grad1} instead of Proposition \ref{corollary regularity}.
\end{proof}

\begin{remark}
The only point in the proof of the previous lemma where we use $0 < \lambda \leq \lambda^{\ast}$ is to show that $T_2(E_2) \subset E_2$. The rest holds for every $\lambda \in \R$.
\end{remark}

\begin{lemma} \label{T2 continuous}
Assume that \eqref{B1} holds. Then $T_2$ is continuous.
\end{lemma}

\begin{proof}
Let $\{\varphi_n\} \subset E_2$ be a sequence such that $\varphi_n \to \varphi$ in $W_0^{s,1}(\Omega)$ and define $u_n = T_2(\varphi_n)$, for all $n \in \N$, and $u = T_2(\varphi)$. We shall show that $u_n \to u$ in $W_0^{s,1}(\Omega)$. Observe that $w_n = u_n -u $ satisfies
\begin{equation} \label{problem wn section 5}
\left\{
\begin{aligned}
(-\Delta)^s w_n & = \mu(x) \left( |(-\Delta)^{\frac{s}{2}} \varphi_n|^q - |(-\Delta)^{\frac{s}{2}} \varphi|^q \right)\,, & \quad \textup{ in } \Omega,\\
w_n & = 0\,, & \quad \textup{ in } \RN \setminus \Omega.
\end{aligned}
\right.
\end{equation}
Hence, if we show that
\begin{equation}
\mu(x) \left( |(-\Delta)^{\frac{s}{2}} \varphi_n|^q - |(-\Delta)^{\frac{s}{2}} \varphi|^q \right) \to 0\,, \quad \textup{ in } L^1(\Omega),
\end{equation}
the result follows from Proposition \ref{convergence}. Directly, since $\varphi_n, \varphi \in E_2$ and $\mu \in \Linfty$, applying the Mean Value Theorem and H\"older inequality, we deduce that
\begin{equation} \label{ineq L1 section 5}
\left\| \mu(x) \left( |(-\Delta)^{\frac{s}{2}} \varphi_n|^q - |(-\Delta)^{\frac{s}{2}} \varphi|^q \right) \right\|_{L^1(\Omega)} \leq C \left( \int_{\Omega} |(-\Delta)^{\frac{s}{2}}(\varphi_n - \varphi)|^q dx \right)^{\frac{1}{q}},
\end{equation}
where $C$ is a positive constant depending only on $\|\mu\|_{L^{\infty}(\Omega)},\, l,\, q$ and $\Omega$. By \eqref{ineq L1 section 5}, if we show that
\begin{equation} \label{convergenceL1 section 5}
\int_{\O} |(-\Delta)^{\frac{s}{2}}(\varphi_n - \varphi)|^q dx \to 0,
\end{equation}
the continuity of the operator follows from Proposition \ref{convergence}.
\medbreak
Since $\varphi_n \to \varphi$ in $W_0^{s,1}(\Omega)$, it follows that $\varphi_n - \varphi \to 0$ almost everywhere in $\Omega$. Furthermore, observe that, for all measurable subset $\omega \subset \Omega$,  we have that
\[ \int_{\omega} |(-\Delta)^{\frac{s}{2}}(\varphi_n - \varphi)|^q dx \leq 2\,l\, |\omega|^{\frac{r-q}{q}}.\]
Hence, by Vitali's convergence Theorem, \eqref{convergenceL1 section 5} holds and the result follows.
\end{proof}

\begin{proof}[\textbf{Proof of Theorem \ref{main th Qlambda}}]
Since $E_2$ is a closed convex set of $W_0^{s,1}(\Omega)$ and, by Lemmas \ref{T2 wd subset compact} and \ref{T2 continuous}, we know that $T_2$ is continuous, compact and satisfies $T_2(E_2) \subset E_2$, we can apply the Schauder fixed point Theorem to obtain $u \in E_2$ such that $T_2(u) = u$. Thus, we conclude that \eqref{Qlambda} has a weak solution for all $0 < \lambda \leq \lambda^{\ast}$.
\end{proof}

\begin{proof}[\textbf{Proof of Theorem \ref{maint th Rlambda}}]
Having at hand Corollary \ref{corollary regu-grad1}, the result follows arguing as in Theorem \ref{main th Qlambda}.
\end{proof}

\section{Further results and open problems} \label{7}

\noindent We end the paper making some remarks and pointing out some possible extensions of our results.
\subsection{Further results} $ $ \medbreak
\begin{itemize}
\item[1)] In the spirit of the existence results of Section \ref{4}, we can deal with more general nonlocal ``gradient terms''. Actually, we can consider a problem of the form
\begin{equation*}
\left\{
\begin{aligned}
(-\Delta)^s u & = \mu(x)\, (\mathbb{B}_s^{q}(u))^{\alpha} + \lambda f(x)\,, & \quad \textup{ in } \Omega,\\
u & = 0\,, & \quad \textup{ in } \RN \setminus \Omega,
\end{aligned}
\right.
\end{equation*}
\noindent where $1 < \alpha \leq q$, $f$ belongs to a suitable Lebesgue space, $\mu \in \Linfty$ and $\mathbb{B}_s^q$ is given by
\begin{equation*}
\mathbb{B}_s^q (u) = \left( \frac{a_{N,s}}{q}\textup{ p.v.}  \int_{\RN} \frac{|u(x)-u(y)|^q}{|x-y|^{N+sq}} dy \right)^{\frac{1}{q}} \,.
\end{equation*}
The existence of a solution for $\lambda f$ small enough can be obtained. \medbreak

\item[2)] On the line of Section \ref{6}, we can consider a problem of the form
\[
\left\{
\begin{aligned}
(-\Delta)^s u & = \mu(x)|(-\Delta)^{\frac{t}{2}} u|^q + \lambda f(x)\,, & \quad \textup{ in } \Omega,\\
u & = 0\,, & \quad \textup{ in } \RN \setminus \Omega,
\end{aligned}
\right.
\]
under the assumptions
\[
\left\{
\begin{aligned}
& \Omega \subset \RN,\ N \geq 2, \textup{ is a bounded domain with }\partial \Omega \textup{ of class } \mathcal{C}^{2}, \\
& f \in L^m(\Omega) \textup{ for some } m \geq 1 \textup{ and } \mu \in \Linfty,\\
& s \in (1/2,1),\  t \in (0,s],\ \textup{ and }\, 1 < q < \frac{N}{N-m(2s-t)}.
\end{aligned}
\right.
\]
A similar result to Theorem \ref{main th Qlambda} can be obtained.
\end{itemize}

\subsection{Open problems} $ $
\medbreak
\begin{itemize}
\item[1)] The Calder\'on-Zygmund type regularity results proved in Section \ref{3} rely on Lemma \ref{AP gradient regularity}, i.e. on \cite[Lemma 2.15]{A_P_2018}. The restriction $s \in (1/2,1)$ comes from this result. It is an open question if the regularity results of Section \ref{3} hold for $s \in (0,1/2]$. 
\smallbreak
Let us also stress that, if the corresponding regularity results with $s \in (0,1/2]$ were available, our approaches to prove Theorems \ref{main th existence},   \ref{main th existence map1},  \ref{main th Qlambda} and \ref{maint th Rlambda} would directly provide the corresponding results.

\medbreak
\item[2)] In the last years there has been a renewed interest in classical problems of the form
\[ -\Delta u = c(x)u + \mu(x)|\nabla u|^2 + h(x), \quad u \in H_0^1(\Omega) \cap \Linfty.\]
Following \cite{J_S_2013, S_2010}, several works have appeared proving existence and multiplicity results. Does this kind of results hold in the nonlocal case$\,$? To be more precise, let us introduce the Dirichlet problem 
\[
\left\{
\begin{aligned}
(-\Delta)^s u & = c(x) u + \mu(x)\,\mathbb{D}_s^2(u) + \lambda f(x)\,, & \quad \textup{ in } \Omega,\\
u & = 0\,, & \quad \textup{ in } \RN \setminus \Omega,
\end{aligned}
\right.
\]
under the assumption \eqref{A1} and $c \in \Linfty$. It seems interesting to address the following questions: \smallbreak
\begin{itemize}
\item[a)] Does the uniqueness of (smooth) solutions holds for $c(x) \leq 0\,$?
\item[b)] Under the assumption $c(x) \leq \alpha_0 < 0$ a.e. in $\Omega$. It is possible to remove the smallness condition imposed on $\lambda\,$?
\item[c)] It is possible to prove the existence of more than one solution for $c(x) \gneqq 0$, $\mu(x) \geq \mu_1 > 0$ and $\lambda > 0$ small enough$\,$?
\end{itemize}
\end{itemize}

\bibliographystyle{plain}
\bibliography{Bibliography}

\end{document}